\documentclass{amsart}
\usepackage{amsmath}
\usepackage{newtxmath}
\usepackage{tikz-cd}
\usepackage{fix-cm}
\usepackage{hyperref}
\usepackage{subfigure}

\newtheorem{thm}{Theorem}[section]
\newtheorem{defn}[thm]{Definition}
\newtheorem{prop}[thm]{Proposition}
\newtheorem{lem}[thm]{Lemma}

\DeclareMathOperator{\arcch}{arccosh}
\DeclareMathOperator{\arcth}{arctanh}
\DeclareMathOperator{\sech}{sech}

\DeclareMathOperator{\interior}{int}
\DeclareMathOperator{\image}{Im}
\DeclareMathOperator{\hess}{Hess}
\DeclareMathOperator{\area}{Area}

\begin{document}

\title[The existence of circle packing on hyperbolic surface]
{The existence of inversive distance circle packing on hyperbolic polyhedral surface}
\author{Xiang Zhu}
\address{Department of mathematics, Shanghai University, Shanghai, China, 200444.}
\email{zhux@shu.edu.cn}

\begin{abstract}
    In this paper, we prove that given a hyperbolic polyhedral metric with an inversive distance circle packing, and a target discrete curvature satisfying Gauss-Bonnet formula, there exist a unique inversive distance circle packing which is discrete conformal to the former one. We deform the surface by discrete Ricci flow, and do surgery by edge flipping when the orthogonal circles of some faces are about to be non-compact. The revised weighted Delaunay inequality of hyperbolic case implies the compactness of the orthogonal circle. We use a variational principle of a convex Ricci potential defined on the fiber bundles with cell-decomposition and differential structure based on Teichm\"uller space to finish the proof.
\end{abstract}

\maketitle

\section{Introduction}

\subsection{Statement of results}

The \emph{inversive distance} is the generalization to the cosine of the intersection angle of two circles. The \emph{inversive distance circle packing} is a polyhedral surface with disjoint circles (or cones precisely) centered at vertices (or cone points). As we know that the conformal mapping on Riemannian surface keeps the infinitesimal circle, given a (geometric) triangulation on this surface, if we adjust the radii of circles while fixing all the inversive distance with respect to edges, then that is an analog or discretization of a conformal map. A hyperbolic polyhedral surface is a two-dimensional Riemannian manifold with $-1$ Gauss curvature, except for finite vertices. The local metric near vertices is a hyperbolic cone. The \emph{discrete curvature} is $2\pi$ minus the cone angle.

This paper is a continuation of \cite{zhu2023existence}, and also a part of translation of the author's thesis \cite{zhu2019discrete}, in which we proved that
\begin{thm}
    Given a hyperbolic polyhedral surface with an inversive distance circle packing, for any target discrete curvature satisfying Gauss-Bonnet formula, there exist a unique one discrete conformal to the initial one.
\end{thm}
Here the definition of discrete conformal only allows radius changing on the weighted Delaunay triangulation. If the weighted Delaunay condition is about to break, we have a generalized Ptolemy equation for the triangulation to do surgery of edge flipping. This theorem will be claimed precisely again after notations in Theorem \ref{main_h}.

The difference between this paper and \cite{zhu2023existence} is just like the difference between \cite{gu2018discreteII} and \cite{gu2018discrete}. Gu et al. proved the discrete conformal deformation on the vertex scaling case, which depends on the compactness of circumcise in $\mathbb{H}^2$. Similarly, we have to deal with the compactness of orthogonal circle.

\subsection{Related work}
An introduced of circle packing is presented in \cite{zhu2023existence}.
A groundbreaking work on circle packing is proposed by Thurston in \cite{thurston1979geometry}. Then there are related works in \cite{chow2003combinatorial}, \cite{marden1990on}, \cite{stephenson2005introduction}, \cite{bobenko2004variational}, \cite{zhang2014unified}, \cite{bobenko2015discrete}, \cite{gu2018discrete}, \cite{gu2018discreteII}, \cite{sun2015discrete}, \cite{penner1987the}, \cite{rivin1994euclidean}, \cite{bowers2004uniformizing}, \cite{bowers2003planar}, \cite{guo2011local}, \cite{luo2011rigidity}, \cite{springborn2008variational}, \cite{ge2017deformation}, \cite{ge2017deformationII}, \cite{ge2019deformation}, \cite{bobenko2023decorated}, \cite{rodin1987convergence}, \cite{stephenson1999approximation}, \cite{rodin1987convergence}, \cite{stephenson1999approximation}, \cite{he1996convergence}, \cite{gu2019convergence}, \cite{luo2021convergence},\cite{chen2022bowersstephensons}, etc. We notice that there are some related work by Bobenko and Lutz in \cite{bobenko2023decorated2} recently. Due to the limitations of the author's knowledge, only a partial list of work is provided above. Please forgive the incompleteness.

\subsection{Organization of the paper}

In Section \ref{section2}, we recall the related definitions, notations and propositions, some of them are from the previous paper \cite{zhu2023existence}.

In Section \ref{section3}, we study formulas on orthogonal circle, whose compactness is as important as triangle inequalities in Euclidean case.

In Section \ref{section4}, we proved diffeomorphism between fiber bundles with cell decomposition based on Teichm\"uller spaces. Since the computation is complicated, we attach the codes.

In Section \ref{section5}, we define the discrete conformal equivalent class as a fiber, and prove the main theorem by a variational principle and then the discrete Ricci flow.

\section{Preliminary and notation}\label{section2}

\subsection{Triangulation}

Suppose $S$ is a closed surface, and given a non-empty finite subset $V \subset S$, then we call the tuple $(S,V)$ a \emph{marked surface}.

In this paper, we use $\Delta$-complex introduced in \cite{hatcher2002algebraic} instead of simplicial complex to define a triangulation, which allows the cells gluing itself. Given a marked surface $(S,V)$, a (topological) triangulation $\mathcal{T}$ is a $\Delta$-complex decomposition up to an isotopy fixing $V$, whose set of $0$-cells is $V$.

We only consider the marked surface with the negative Euler characteristic when punched, which is $\chi(S \setminus V)<0$. If not, there is no triangulation for it. Also, the marked surface should be oriented, or we would study its oriented double cover instead. 

We denote $e_{ij}$ to be the $1$-cell of $\mathcal{T}$, or the \emph{edge}, connecting $v_i$ and $v_j$, and $f_{ijk}$ to be the $2$-cell of $\mathcal{T}$, or the \emph{face}, surrounded by $e_{ij}$, $e_{ik}$ and $e_{jk}$. Denote the number of vertices $\left|V\right|$ by $n$, and the genus of $S$ by $g$. We use $\mathbb{R}^A$ to represent the space of function $\{\,\mathbf{x} \colon A \to \mathbb{R} \mid \mathbf{x}(a_i)=x_i \in \mathbb{R}\,\}$ defined on a finite set domain $A=\{a_1,\dots,a_m\}$. Here $A$ might be $V$ or $E$.

We denote a \emph{hinge} $\Diamond_{ij;kl}$ to be the edge $e_{ij}$ with two faces $f_{ijk}$ and $f_{ijl}$ at both sides. A \emph{flipping} of $e_{ij}$ is to replace the hinge $\Diamond_{ij;kl}$ by $\Diamond_{kl;ij}$. There are infinitely many triangulations on a marked surface in general, however, for any two triangulations of $(S,V)$, one can transform to another one by a finite number of flipping \cite{hatcher1991triangulations}\cite{Mosher1988Tiling}\cite{penner1987the}.

Given a marked surface $(S,V)$, if for any point $p\in S\setminus V$ there exists a neighborhood isometric to a region in $\mathbb{H}^2$, and for any vertex $v_i \in V$ there exists a neighborhood isometric to a region in a \emph{hyperbolic cone} with a cone angle of $\varphi_i$, which is the set
\begin{equation}%
    \left\{\, (r,\theta) \mid 0 \le r < 1, \theta \in \mathbb{R}/\varphi\mathbb{Z} \,\right\} / (0,\theta_1) \sim (0,\theta_2)
\end{equation}%
with the metric
\[
    \frac{4(dr^2+r^2d\theta^2)}{(1-r^2)^2},
\]
then we say that $(S,V)$ has a \emph{hyperbolic polyhedral metric} or a \emph{piecewise hyperbolic metric}, denoted by $d_h$, where $h$ means \emph{hyperbolic}. A marked surface with a metric $(S,V,d_h)$ is called a \emph{hyperbolic polyhedral surface} or a \emph{piecewise hyperbolic surface}. For any points $p,q \in S$, we use $d_h(p,q)$ to indicate the hyperbolic distance between $p$ and $q$. 

The \emph{discrete curvature} at $v_i$ is $K_i \coloneqq 2\pi-\varphi_i < 2\pi$. The Gauss-Bonnet formula is  $\sum_{i=1}^{n}K_i=2\pi \chi(S)+\area(S)$, where $\area$ means the hyperbolic area of the region.

Any hyperbolic polyhedral surface $(S,V,d_h)$ has a \emph{geodesic triangulation}, that is, a $\Delta$-complex decomposition with all edges geodesic on $d_h$. For example, the Delaunay triangulation is a geodesic triangulation for a hyperbolic polyhedral surface \cite{gu2018discreteII}.

\subsection{Teichm\"uller space}

The \emph{Teichm\"uller space of hyperbolic polyhedral metric} on $(S,V)$, denoted by $Teich_h(S,V)$, is the space of all hyperbolic polyhedral metric on $(S,V)$ considered up to isometry isotopic to the identity map fixing $V$. The isometry class of a hyperbolic polyhedral metric $d_h$ is denoted by $[d_h]$.

Given a triangulation $\mathcal{T}$ of $(S,V)$, the edge length of $\mathcal{T}$ should be in the convex polytope
\[
    \mathbb{R}^{E(\mathcal{T})}_\Delta \coloneqq
    \left\{\, \mathbf{x} \in \mathbb{R}^{E(\mathcal{T})} \mid
    \forall f_{ijk} \in F(\mathcal{T}),\,
    0<x(e_{ij})<x(e_{jk})+x(e_{ki}) \,\right\}
\]
Given a length function $\mathbf{x} \in \mathbb{R}^{E(\mathcal{T})}_\Delta$, replace $f_{ijk}$ by a hyperbolic triangle of edge lengths $\mathbf{x}(e_{ij}),\mathbf{x}(e_{jk}),\mathbf{x}(e_{ki})$, and glue them by isometries along the corresponding edges, then we construct a piecewise flat metric $d_\mathbf{x}$ and produce an injective map
\[
    \Psi_\mathcal{T} \colon \mathbb{R}^{E(\mathcal{T})}_\Delta
    \to Teich_h(S,V) \quad \mathbf{x} \mapsto [d_\mathbf{x}].
\]
Let $P_h(\mathcal{T}) \coloneqq \Psi_\mathcal{T}(\mathbb{R}^{E(\mathcal{T})}_\Delta)$, we have
\[
    Teich_h(S,V)=\bigcup_\mathcal{T}P_h(\mathcal{T}), 
\]
where the union is over all triangulations of $(S,V)$. 

Similar to $Teich_h(S,V)$ in \cite{zhu2023existence}, we can prove that $Teich_h(S,V)$ is a simply connected real analytic manifold in the same way by using hyperbolic cosine law \eqref{cos} and pentagon relation.

Given a surface $\Sigma$ with or without boundary, the \emph{Teichm\"uller space of hyperbolic metric} on $\Sigma$, denoted by $Teich(\Sigma)$, is the space of all hyperbolic metric with closed geodesic boundary on $\Sigma$ considered up to isometry isotopic to the identity map. We denote a hyperbolic metric by $d$ and its isotopy class by $[d]$.

In this paper we only consider the case that $\Sigma$ is oriented, compact with boundaries, and $\chi(\Sigma)<0$. By gluing a topological open dist $D_i$ on every boundary of $\Sigma$, and selecting any point in $D_i$ as the vertex $v_i$, we get a closed marked surface $(S,V)$. The inverse operation is also legal. We call that $(S,V)$ and $\Sigma$ are \emph{related}. A \emph{truncated triangulation} of $\Sigma$ is the intersection of $\Sigma$ and a triangulation of its related surface $(S,V)$, denoted by $\mathcal{T}=(\Gamma,E,F)$, where $\Gamma=\{\,\gamma_i=\partial D_i,\,i=1 \dots n\,\}$ and $n=\left|V\right|$. For any hyperbolic metric $d$ on $\Sigma$, there exist a unique geodesic isotopic to $e_i \in E$ and orthogonal to the geodesic boundary $\partial \Sigma$ \cite{luo2007teichmuller} \cite{dai2008variational}.

\begin{defn}\label{Oemga}	
    Given a hyperbolic surface with $(\Sigma,d)$ and a truncated triangulation $\mathcal{T}$, we define a length function    
    \[
        \mathbf{x} \colon E(\mathcal{T}) \to \mathbb{R}_{>0}
        \quad e_i \mapsto x_i
    \]
    or $\mathbf{x} \in \mathbb{R}_{>0}^{E(\mathcal{T})}$, called the \emph{length coordinate} of $d$, where $x_i$ is the length of the unique geodesic isotopic to $e_i$ and orthogonal to $\partial \Sigma$. We define the coordinate chart with respect to $\mathcal{T}$ by
    \[
        \Omega_\mathcal{T}^{-1} \colon Teich(\Sigma) \to \mathbb{R}_{>0}^{E(\mathcal{T})}
        \quad [d] \mapsto \mathbf{x}.
    \]
\end{defn}

\begin{thm}\label{teich home}		
    $\Omega_\mathcal{T}$ is a homeomorphic map \cite{dai2008variational}\cite{ushijima1999canonical}.
\end{thm}

In paper \cite{zhu2023existence}, Zhu introduced the \emph{discriminant of inversive distance} 
\[
    \Delta_{abc} \coloneqq a^2+b^2+c^2+2abc-1,
\]
and proved the following theorem.

\begin{thm}\label{teich home formula}
    Suppose $\mathcal{T}$ is a triangulation of $(S,V)$, and we get another triangulation $\mathcal{T}'$ by flipping the diagonal of the hinge $\Diamond_{ij;kl}$, then the transform map between these two coordinate charts is
    \[
        \Omega_{\mathcal{T}'}^{-1} \circ \Omega_\mathcal{T} \colon
        \mathbb{R}_{>0}^{E(\mathcal{T})} \to \mathbb{R}_{>0}^{E(\mathcal{T}')}
        \quad \mathbf{l} \mapsto \mathbf{l}'
    \]
    \[			
        \mathbf{l}'(e_m)=			
        \begin{cases}
            \mathbf{l}(e_m) & e_m \ne e_{kl} \\
            \arcch f & e_m = e_{kl}		
        \end{cases}
    \]
    where $a=\cosh \mathbf{l}(e_{ki}),b=\cosh \mathbf{l}(e_{il}),c=\cosh \mathbf{l}(e_{lj}),d=\cosh \mathbf{l}(e_{jk}),e=\cosh \mathbf{l}(e_{ij})$ and
    \begin{equation}\label{f}
        f = \frac{ab+cd+ace+bde+\sqrt{\Delta_{ade}}\sqrt{\Delta_{bce}}}{e^2-1},
    \end{equation}
    See Figure \ref{fig:4hyp}.

    For any two different triangulation of $(S,V)$, we can transform from one to another by finite steps of flipping, and the corresponding transform map is the composition of the transform maps with respect to the flipping.
\end{thm}

\begin{figure}[ht]
    \centering
    \includegraphics{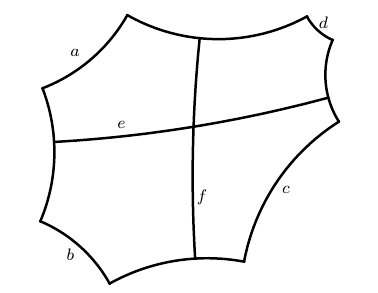}
    \caption{Notation in a hyperbolic right angle octagon related with a flipping hinge.}
    \label{fig:4hyp}
\end{figure}    

The \emph{generalized Ptolemy equation} below is the quadratic equation to which $f$ in the theorem above satisfies.
\begin{equation}\label{abcdef}
    \begin{aligned}
        a^2 + b^2 + c^2 + d^2 + e^2 + f&^2\\
        + 2 (ade + bce + abf + cdf + abcd + acef + bdef&)\\
        - a^2 c^2 - b^2 d^2 - e^2 f^2 - 1& = 0.
    \end{aligned}
\end{equation}
Moreover, we have
\begin{equation}\label{delta}
    \begin{aligned}
    \sqrt{\Delta _{a b f}}&=\frac{(d+a e) \sqrt{\Delta _{b c e}}+(c+b e) \sqrt{\Delta _{a d e}}}{e^2-1},\\	
    \sqrt{\Delta _{c d f}}&=\frac{(a+d e) \sqrt{\Delta _{b c e}}+(b+c e) \sqrt{\Delta _{a d e}}}{e^2-1}.
    \end{aligned}
\end{equation}

\subsection{Weighted Voronoi-Delaunay duality}
In the paper \cite{zhu2019discrete}, Zhu proved that weighted Voronoi decomposition on hyperbolic polyhedral surface is a unique CW decomposition by revising the definition of Voronoi cell. Also, he proved that the existence and uniqueness of weighted Delaunay triangulation on a hyperbolic polyhedral surface by using a construction of the isotopy cover map.

Given a hyperbolic polyhedral surface $(S,V,d_h)$, define radii by $\mathbf{r} \colon V \to \mathbb{R}_{>0},\, v_i \mapsto r_i$ or $\mathbf{r} \in \mathbb{R}_{>0}^V$ as weight, the domain of the weight is defined to be
\[
    R = \{\,\mathbf{r} \in \mathbb{R}_{>0}^V \mid
    0<r_i<\mathrm{Inj}(v_i) \mbox{ and }
    \forall i \ne j,r_i+r_j<d_h(v_i,v_j) \,\}, 
\]
where $\mathrm{Inj}$ indicates the injective radius.

\begin{defn}	
    Given a hyperbolic polyhedral surface $(S,V,d_h)$, and a weight function $\mathbf{r} \in R$, the \emph{inner weighted Voronoi cell} of $v_i$, denoted by $Vor_h(v_i)$, is defined to be the set of all $p\in S$ satisfying that 
    \begin{itemize}
        \item there exists a unique shortest geodesic on $S$ connecting $p$ and $v_i$, and
        \item for any $j\neq i$,
        \[
            \frac{\cosh d_h(p,v_i)}{\cosh r_i}<\frac{\cosh d_h(p,v_j)}{\cosh r_j}.
        \]
    \end{itemize}    
\end{defn}

\begin{thm}\label{vor_f}		
    There exists a unique CW decomposition of $S$, called the \emph{weighted Voronoi decomposition}, of which the set of all the $2$-cells is $\{\, Vor_h(v_i) \mid v_i \in V \,\}$.
\end{thm}

\begin{defn}\label{del_def}
    Given a hyperbolic polyhedral surface $(S,V,d_h)$ with weight $\mathbf{w} \in W$, there exist a unique CW decomposition whose $1$-cells are geodesics connecting vertices, called \emph{weighted Delaunay tessellation}, to be the dual graph of the weighted Voronoi decomposition.

    The \emph{weighted Delaunay triangulation} is to subdivide hyperbolic polygon faces of the weighted Delaunay tessellation into triangles by connecting some geodesic diagonals in any ways.
\end{defn}

Suppose the corresponding weighted Delaunay triangulation is $\mathcal{T}=(V,E,F)$, for any face $f_{ijk} \in F$, we denote the dual of $f_{ijk}$ on $S$ by $O_{ijk}$, which is a $0$-cell of the weighted Voronoi decomposition. 

There is a circle, called the \emph{orthogonal circle}, denoted by $\odot O_{ijk}$, orthogonal to three circles centered at $v_i,v_j,v_k$ with radius $r_i$, $r_j$ and $r_i$, and centered at $O_{ijk}$ with radius $\rho_{ijk}>0$, such that
\[
    \cosh\rho_{ijk}=\frac{\cosh d_h(O_{ijk},v_i)}{\cosh r_i}
    =\frac{\cosh d_h(O_{ijk},v_j)}{\cosh r_j}=\frac{\cosh d_h(O_{ijk},v_k)}{\cosh r_k}
\]

Note that these notations above also hold when $\mathcal{T}$ is only geodesic but not weighted Delaunay, except that $O_{ijk}$ may not exist on $S$. For this case we just immerse the local part in to $\mathbb{E}^2$ to find $O_{ijk}$.

\begin{defn}\label{loc_del}
    Given a face $f_{ijk}$, we denote the distance between $O_{ijk}$ and edge $e_{ij}$ by $h_{ij,k}$, which is positive when $O_{ijk}$ and $v_k$ are on the same side of $e_{ij}$, and negative when they are on the different sides.
    
    An edge $e_{ij}$ is called to be \emph{local weighted Delaunay} if $h_{ij,k} + h_{ij,l} \ge 0$.
\end{defn}

\begin{thm}\label{local_whole_h}
    Given a  piecewise flat surface $(S,V,d_h)$ with a weight function $\mathbf{r} \in R$, we know that:
    \begin{itemize}
        \item The weighted Delaunay triangulation exists.
        \item The triangulation is weighted Delaunay, if and only if all the edges are local weighted Delaunay.
        \item The weighted Delaunay triangulation is unique up to finite diagonal switches.
    \end{itemize}
\end{thm}

The first and third statements are obviously true based on the analog of Euclidean case. For the second statement, the methods presented in the papers \cite{bobenko2007discrete} and \cite{gorlina2011weighted} can be directly generalized to the hyperbolic case.

We recall the notations mentioned in \cite{zhu2023existence} on the weighted Voronoi-Delaunay duality on hyperbolic surface with geodesic boundaries.

\begin{lem}\label{gbh_sinh}
    Let $r_1,r_2>0$ and $\gamma_1,\gamma_2$ be two geodesics that do not intersect in $\mathbb{H}^2$, then the set of all points $q \in \mathbb{H}^2$ satisfied
    \begin{equation}\label{vor_sinh}
        r_1 \sinh d(q,\gamma_1)=r_2 \sinh d(q,\gamma_2),
    \end{equation}
    is a geodesic, which is orthogonal to the geodesic segment connecting $\gamma_1,\gamma_2$.
\end{lem}

\begin{defn}
    Given a hyperbolic surface $(\Sigma,d)$ with geodesic boundaries $\partial \Sigma=\{\gamma_i,\dots,\gamma_n\}$ and a weight function $\mathbf{r} \in \mathbb{R}^V_{>0}$, 
    the \emph{inner weighted Voronoi cell} of $v_i$ in this case, denoted by $Vor(\gamma_i)$, is defined to be the set of all $p \in S$ such that there exists a unique geodesic connecting $p$ and $\gamma_i$ whose length is the distance $d(p,\gamma_i)$, and for any $j \ne i$, the inequality
    \[
        r_i \sinh d(v_i,\gamma_i)<r_j \sinh d(v_j,\gamma_j)
    \]
    holds.
\end{defn}

\begin{thm}\label{gbh_vor}
    Let $(S,V)$ be the related surface of $\Sigma$, namely, $\Sigma = S \setminus \bigcup_{i=1}^n D_i$ where $D_i$ are disjoint open disks containing $v_i$. Given a weight function $\mathbf{r} \in \mathbb{R}^V_{>0}$, there exists a CW decomposition of $S$, called \emph{weighted Voronoi decomposition}, with $\{\, Vor(\gamma_i) \cup D_i \mid i=1,\dots,n \,\}$ as $2$-cells. This decomposition is unique restricted on $\Sigma$.
\end{thm}

\begin{thm}\label{gbh_unique}
    There exist a unique geodesic \emph{truncated} triangulation of $(\Sigma,d)$ up to finite diagonal switches with respect to $\mathbf{r} \in \mathbb{R}^V_{>0}$, called the \emph{weighted Delaunay triangulation}, to be the dual graph of the weighted Voronoi decomposition adding some geodesic diagonals into truncated polygons other than truncated triangle.
\end{thm}

\subsection{Inversive distance circle packing}

In the paper \cite{guo2011local}, Guo proved the local rigidity of inversive distance circle packing. We adopt the notations from this paper.

Recall that the inversive distance for two circles with the radii $r_1,r_2>0$ and the distance between their centers $l_{12}$ in $\mathbb{E}^2$, then their inversive distance is
\begin{equation}\label{inveuc}
    I_{12} \coloneqq \frac{l_{12}^2-r_1^2-r_2^2}{2r_1 r_2}.
\end{equation}
For the hyperbolic case, the definition is different.

\begin{defn}
    Given two circles in $\mathbb{H}^2$ with radii $r_1,r_2>0$, and the hyperbolic distance between their centers is $l_{12}$, then their \emph{inversive distance} is defined by
    \begin{equation}\label{invhyp}
        I_{12} \coloneqq \frac{\cosh l_{12}-\cosh r_1 \cosh r_2}{\sinh r_1 \sinh r_2}
    \end{equation}
\end{defn}
 
The definition also holds for two circles with centers located at vertices on a piecewise flat surface. If there is a geodesic connecting these two vertices without passing through other vertices, the definition remains valid. We only care the case of $I_{12}>1$ or $r_1+r_2<l_{12}$, and if $I_{12} \le 1$, we should study the intersection angle instead. For a geodesic that connects the same vertex without passing through other vertices, if the radius of the circle at that vertex is smaller than the injective radius at that point, it can be considered as two disjoint circles located at the two ends of this geodesic. Similarly, the inversive distance can be defined as well.

Given a hyperbolic polyhedral surface $(S,V,d_h)$ with a geodesic triangulation $\mathcal{T}=(V,E,F)$, suppose $\mathbf{I} \in \mathbb{R}_{>1}^{E(\mathcal{T})}$ and $\mathbf{r} \in \mathbb{R}_{>0}^V$, then we call the collection of every geodesic circles with radius $r_i$ centered at vertex $v_i \in V$ on $S$ as an \emph{inversive distance circle packing} with respect to the inversive distance $\mathbf{I}$, if every couple of circles between the geodesic edge $e_{ij} \in E$ has the inversive distance $I_{ij} \in \mathbf{I}$.

We give the inversive distance circle packing another notation. Given a  marked surface $(S,V)$ with a topological triangulation $\mathcal{T}$, we define the map
\[
    L_h \colon \mathbb{R}_{>1}^{E(\mathcal{T})} \times
	\mathbb{R}_{>0}^V \to \mathbb{R}_{>0}^{E(\mathcal{T})}
	\quad (\mathbf{I},\mathbf{r}) \mapsto \mathbf{l}
\]
where
\[
    \mathbf{l} \colon E(\mathcal{T}) \to \mathbb{R}_{>0} \quad
    e_{ij} \mapsto \arcch(\cosh r_i \cosh r_j + I_{ij} \sinh r_i \sinh r_j).
\]

If $L_h(\mathbf{I},\mathbf{r}) \in \mathbb{R}^{E(\mathcal{T})}_\Delta$, or the image of the map satisfies every triangle inequalities on all faces of $\mathcal{T}$, we take an element of the isotopy class of $\Phi_\mathcal{T} \circ L_f(\mathbf{I},\mathbf{r})$, called the hyperbolic polyhedral metric $d_h$. Then the triangulation $\mathcal{T}$ becomes a geodesic triangulation of the hyperbolic polyhedral $(S,V,d_f)$. We say that $(\mathbf{I},\mathbf{r})$ is an inversive distance circle packing of $(S,V,d_h)$ with respect to $\mathcal{T}$.

Note that $L_h$ is related to the triangulation, however, we omit the symbol $\mathcal{T}$ on the notation on the mapping, since it appears in the notation of the domain $\mathbb{R}_{>1}^{E(\mathcal{T})} \times \mathbb{R}_{>0}^V$.

At the end of this section, we define the \emph{lifting} of $L_h$. Denote that
\[
    Q_h(\mathcal{T}) \coloneqq L_h^{-1}(\mathbb{R}^{E(\mathcal{T})}_\Delta)
    \subset \mathbb{R}_{>1}^{E(\mathcal{T})} \times \mathbb{R}_{>0}^V.
\]

Define the lifting map
\[
    \tilde{L}_h \colon Q_h(\mathcal{T}) \to 
    \mathbb{R}^{E(\mathcal{T})}_\Delta \times \mathbb{R}_{>0}^V
    \quad (\mathbf{I},\mathbf{r}) \mapsto (L_h(\mathbf{I},\mathbf{r}),\mathbf{r})
\]
and the projection map
\[
    \pi \colon \mathbb{R}_{>0}^{E(\mathcal{T})} \times \mathbb{R}_{>0}^V
    \to \mathbb{R}_{>0}^{E(\mathcal{T})} \quad (\mathbf{l},\mathbf{r}) \mapsto \mathbf{l}\,,
\]
then we have
\[
    \pi \circ \tilde{L}_h = L_h.
\]

\begin{prop}\label{lift_inject}
    $\tilde{L}_h$ is an injective map, then a real analytic homeomorphism from its domain to its image.
\end{prop}

\begin{proof}
    If $\tilde{L}_h(\mathbf{I},\mathbf{r})=\tilde{L}_h(\mathbf{I}',\mathbf{r}')$, we have $(L_h(\mathbf{I},\mathbf{r}),\mathbf{r})=(L_h(\mathbf{I}',\mathbf{r}'),\mathbf{r}')$, then $\mathbf{r}=\mathbf{r}'$. Thus, by
    \[
        \cosh r_i \cosh r_j + I_{ij} \sinh r_i \sinh r_j)=
        \cosh r'_i \cosh r'_j + I'_{ij} \sinh r'_i \sinh r'_j)
    \]
    we know that for any $e_{ij} \in E$ and $I_{ij}=I'_{ij}$, or $\mathbf{I}=\mathbf{I}'$.
\end{proof}

\subsection{Cell decomposition}
Given a marked surface $(S,V)$ and its related hyperbolic surface with geodesic boundaries $(\Sigma,d)$, with a weight function $\mathbf{r} \in \mathbb{R}^V_{>0}$ and a truncated triangulation $\mathcal{T}=(\Gamma,E,F)$. Denote that $\mathbf{l}=\Omega_\mathcal{T}^{-1}([d])$ and $l_{ij}=\mathbf{l}(e_{ij})$.

For a truncated hinge $\Diamond_{ij;kl}$, we denote that $p=r_k,q=r_i,r=r_l,s=r_j,a=\cosh l_{ki},b=\cosh l_{il},c=\cosh l_{lj},d=\cosh l_{jk},e=\cosh l_{ij}$ similarly as in Theorem \ref{teich home formula}.

\begin{defn}\label{dft}
    Let $f$ satisfy the formula \eqref{f}, the \emph{local weighted Delaunay inequality} for the edge $e_{ij}$ is defined (in Euclidean case) by
    \begin{equation}\label{ineq_f}
        \frac{\sqrt {\Delta _{bce}}}{p} + \frac{\sqrt {\Delta _{ade}}}{r}\le
        \frac{\sqrt {\Delta _{cdf}}}{q} + \frac{\sqrt {\Delta _{abf}}}{s}
    \end{equation}

    The set of $(\mathbf{I},\mathbf{r})$ satisfied the inequality for edge $e_{ij}$ of the triangulation $\mathcal{T}$, denote by $D_f(e_{ij};\mathcal{T})$, is a closed set in $\mathbb{R}_{>1}^{E(\mathcal{T})} \times \mathbb{R}_{>0}^V$. 
    
    The intersection of all these $\left|E\right|$ number of closed sets with respect to $\mathcal{T}$ is denoted by
        \begin{equation}\label{ineq_f2}
        D_f(\mathcal{T}) \coloneqq \bigcap_{e_{ij} \in E(\mathcal{T})}
        D_f(e_{ij};\mathcal{T}) \subset \mathbb{R}_{>1}^{E(\mathcal{T})}
        \times \mathbb{R}_{>0}^V.
    \end{equation}
\end{defn}

\begin{lem}\label{gbh_del}
    With the notations above, the edge $e_{ij}$ is local weighted Delaunay if and only if the inequality \eqref{ineq_f} holds, where $f$ satisfies the formula \eqref{f}.
\end{lem}

\begin{thm}\label{gbh_whole_del}
    Given a hyperbolic surface with geodesic boundaries $(\Sigma,\Gamma,d)$ and a weight function $\mathbf{r} \in \mathbb{R}^\Gamma_{>0}$, the truncated triangulation $\mathcal{T}$ is weighted Delaunay if and only if there exist $(\mathbf{I},\mathbf{r}) \in D_f(\mathcal{T})$ such that $\Omega_\mathcal{T} \circ \arcch(\mathbf{I})=[d]$, where
    \[
        \arcch \colon \mathbb{R}^k_{\ge 1} \to \mathbb{R}^k_{\ge 0}\quad
        (I_1,\dots,I_k)\mapsto(\arcch I_1,\dots,\arcch I_k).
    \]
\end{thm}

In paper \cite{zhu2023existence}, Zhu claimed that if we denote the map as
\begin{equation}\label{omega_f}
    \tilde{\Omega}_\mathcal{T} \colon \mathbb{R}_{>1}^{E(\mathcal{T})}
    \times \mathbb{R}_{>0}^V \to Teich(\Sigma) \times 
    \mathbb{R}_{>0}^V \quad (\mathbf{I},\mathbf{r}) \mapsto
    (\Omega_\mathcal{T}\circ \arcch(\mathbf{I}),\mathbf{r}),
\end{equation}
then 
\begin{equation}\label{union_teich_gb}
    Teich(\Sigma) \times \mathbb{R}_{>0}^V=\bigcup_\mathcal{T}
    \tilde\Omega_\mathcal{T}(D_f(\mathcal{T}))
\end{equation}
forms an infinite but local finite cell decomposition, where the union is over all truncated triangulations of $\Sigma$.

\section{Compact orthogonal circle}\label{section3}

\subsection{Orthogonal circle formula}

Given an inversive distance function on edges $\mathbf{I} \in \mathbb{R}_{>1}^{E(\mathcal{T})}$ with respect to the triangulation $\mathcal{T}$ and a weight function $\mathbf{r} \in \mathbb{R}_{>0}^V$ as well on a marked surface $(S,V)$, we can construct a hyperbolic polyhedral metric $d_h \in \Psi_\mathcal{T} \circ L_h(\mathbf{I},\mathbf{r})$. 

Firstly, we consider the inversive distance on one single triangle face $f_{ijk}$ embedded in Poincar\'e disk $\mathbb{D} \cong \mathbb{H}^2$. See Figure \ref{fig:orth3h}. We denote that $P=v_i,Q=v_j,R=v_k$, and 
\[
    \begin{aligned}
        a=I_{jk},b=I_{ki},c=I_{ij} &\quad x=\cosh QR,y=\cosh RP,z=\cosh PQ\\
        p=\cosh r_i,q=\cosh r_j,r=\cosh r_k &\quad 
        \hat{p}=\tanh r_i,\hat{q}=\tanh r_j,\hat{r}=\tanh r_k
    \end{aligned}
\]
where $PQ$, $QR$ and $RP$ are the hyperbolic lengths of edge $e_{ij}$, $e_{jk}$ and $e_{ki}$.

\begin{figure}[ht]
    \centering \includegraphics{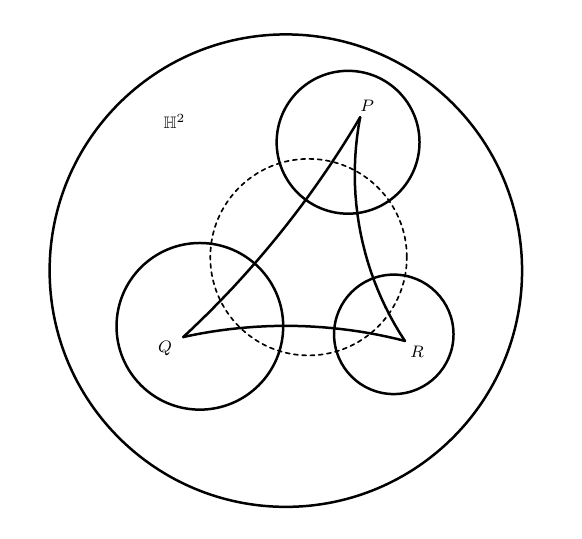}
    \caption{The dashed orthogonal circle is compact.} \label{fig:orth3h}
\end{figure}

If we observe them as the conformally embedding $\mathbb{D} \hookrightarrow \mathbb{E}^2$, then a hyperbolic circle maps to a Euclidean circle. Thus, the image of $\odot P$,$\odot Q$ and $\odot R$ from circle packing are three disjoint circles in $\mathbb{E}^2$. 
They have a unique orthogonal circle, denoted by $\odot O_{ijk}$. If $\odot O_{ijk} \subset \mathbb{D}$ then the orthogonal circle is \emph{compact}. Otherwise, it is \emph{non-compact}.

We define the \emph{discriminant of compactness} 
\begin{equation}\label{xi_xyz}
    \Xi_{ijk} \coloneqq p^2(1-x^2)+q^2(1-y^2)+r^2(1-z^2)
    +2pq(xy-z)+2pr(xz-y)+2qr(yz-x).
\end{equation}

Since $\cosh(\arcth t)=(1-t^2)^{-\frac12}$, substitute
\[
    \begin{aligned}
        x=qr+a\sqrt{1-q^2}\sqrt{1-r^2} &\quad p=(1-\hat{p}^2)^{-\frac12}\\
        y=rp+b\sqrt{1-r^2}\sqrt{1-p^2} &\quad q=(1-\hat{q}^2)^{-\frac12}\\
        z=pq+c\sqrt{1-p^2}\sqrt{1-q^2} &\quad r=(1-\hat{r}^2)^{-\frac12}
    \end{aligned}
\]
into \eqref{xi_xyz} and simplify, we have
\begin{equation}\label{xi_abc}
    \begin{aligned}
        (1-\hat{p}^2)(1-\hat{q}^2)(1-\hat{r}^2)\Xi_{ijk}
        =&(1-c^2)\hat{p}^2\hat{q}^2+(1-b^2)\hat{p}^2\hat{r}^2
        +(1-a^2)\hat{q}^2\hat{r}^2 \\ +&2((a+bc)\hat{p}+(b+ac)\hat{q}
        +(c+ab)\hat{r})\hat{p}\hat{q}\hat{r}.
    \end{aligned}
\end{equation}

Denote that 
\begin{equation}%
    \widehat{PQ}=\sqrt{\hat{p}^2+\hat{q}^2+2c\hat{p}\hat{q}}\quad
    \widehat{PR}=\sqrt{\hat{p}^2+\hat{r}^2+2b\hat{p}\hat{r}}\quad
    \widehat{QR}=\sqrt{\hat{q}^2+\hat{r}^2+2a\hat{q}\hat{r}},
\end{equation}%
then by simply computing,
\begin{equation}\label{xi_trig}
    \Xi_{ijk} = \frac{(\widehat{PQ}+\widehat{PR}+\widehat{QR})
    (\widehat{PQ}+\widehat{PR}-\widehat{QR})
    (\widehat{PQ}+\widehat{QR}-\widehat{PR})
    (\widehat{PR}+\widehat{QR}-\widehat{PQ})}
    {(1-\hat{p}^2)(1-\hat{q}^2)(1-\hat{r}^2)}.
\end{equation}

\begin{thm}\label{xi_thm}
    Given an inversive distance circle packing $L_h(\mathbf{I},\mathbf{r})$ on a hyperbolic polyhedral surface. With the notations above, if the lengths of face $f_{ijk}$ satisfy the triangle inequalities, and there exist a compact orthogonal for it, of which the radius is denoted by $\rho_{ijk} \in \mathbb{R}_{>0}$. Then we have $\Xi_{ijk}>0$ and
    \begin{equation}\label{radius_tanh}
        \sinh r_i \sinh r_j \sinh r_k \sqrt{\Delta_{abc}}
        =\tanh\rho_{ijk}\sqrt{1+2xyz-x^2-y^2-z^2}
        =\sinh\rho_{ijk}\sqrt{\Xi_{ijk}}.
    \end{equation}
\end{thm}

\begin{proof}
    Since the orthogonal circle is compact, it is contained inside $\mathbb{D}$, more over, $\cosh d(O_{ijk},v_\alpha)=\cosh r_\alpha \cosh \rho_{ijk}$ where $\alpha=i,j,k$. In paper \cite{zhang2014unified}, it is known that
    \begin{equation}\label{radius_cosh}
        \cosh^2\rho_{ijk}=\frac{1+2xyz-x^2-y^2-z^2}{\Xi_{ijk}}.
    \end{equation}
    Keep calculating, we have
    \begin{equation}\label{radius_sinh}
        0<\sinh^2\rho_{ijk}=\frac{1+2xyz-x^2-y^2-z^2-\Xi_{ijk}}{\Xi_{ijk}}.
    \end{equation}

    Substituting the inversive distances $a,b,c>1$ and the hyperbolic cosine of radii, compute the numerator, we have
    \begin{equation}\label{xi_delta}		
        1+2xyz-x^2-y^2-z^2-\Xi_{ijk}=(p^2-1)(q^2-1)(r^2-1)\,\Delta_{abc}>0.
    \end{equation}
    Hence, the denominator $\Xi_{ijk}>0$. Note that $\sinh^2(\arcch t)=t^2-1$, then,
    \begin{equation}\label{radius_sinh2}
        \sinh^2\rho_{ijk}=\frac{\sinh^2r_i\sinh^2r_j\sinh^2r_k\,\Delta_{abc}}{\Xi_{ijk}}>0.
    \end{equation}
    Since $r_i,r_j,r_k,\rho_{ijk}>0$, take square root of \eqref{radius_sinh2} on both sides,
    \begin{equation}\label{radius_sinh3}
        \sinh\rho_{ijk}\sqrt{\Xi_{ijk}}
        =\sinh r_i \sinh r_j \sinh r_k \sqrt{\Delta_{abc}}.
    \end{equation}

    To divide \ref{radius_sinh2} by \ref{radius_cosh}, take the square root of both sides, and substitute in \ref{radius_sinh3}, then
    \begin{equation}%
        \tanh\rho_{ijk}\sqrt{1+2xyz-x^2-y^2-z^2}
        =\sinh\rho_{ijk}\sqrt{\Xi_{ijk}}.
    \end{equation}%
\end{proof}

The set of $(\mathbf{I},\mathbf{r})$ satisfying $\Xi_{ijk}>0$ forms an open set in $\mathbb{R}_{>1}^{E(\mathcal{T})} \times \mathbb{R}_{>0}^V$, denoted by $\Xi(f_{ijk};\mathcal{T})$. Define that
\begin{equation}%
    \Xi(\mathcal{T}) \coloneqq \bigcap_{f_{ijk} \in F(\mathcal{T})}
    \Xi(f_{ijk};\mathcal{T}) \subset \ 
    \mathbb{R}_{>1}^{E(\mathcal{T})} \times \mathbb{R}_{>0}^V.
\end{equation}%

\begin{thm}\label{Xi_trig}
    If $\Xi_{ijk}>0$, then $f_{ijk}$ satisfies triangle inequalities, namely, $\Xi(\mathcal{T}) \subset Q_h$. Moreover, the orthogonal circle is compact.
\end{thm}

\begin{proof}
    From Equation \eqref{xi_delta}, we have $2xyz+1-x^2-y^2-z^2>\Xi_{ijk}>0$. Therefore, we can conclude that $(x-yz)^2<(y^2-1)(z^2-1)$. By the definition of $L_h$, we know that $y,z>1$. Hence, $x-yz<\sqrt{y^2-1}\sqrt{z^2-1}$. In other words, $\cosh QR < \sinh PQ \sinh PR + \cosh PQ \cosh PR = \cosh(PQ+PR)$. Since the edge lengths are positive real numbers, and the hyperbolic cosine function is monotonically increasing in $\mathbb{R}_{>0}$, we can conclude that $QR<PQ+PR$. The same reasoning can be applied to prove the other triangle inequalities.

    Since face $f_{ijk}$ satisfies triangle inequalities, by isometric embedded in $\mathbb{D}$ and conformally embedded in $\mathbb{E}^2$, we know that there exist the orthogonal circle in $\mathbb{E}^2$. To prove this theorem, we will use proof by contradiction and consider two cases.

    Suppose that the orthogonal circle is a non-compact horocycle, or is tangent to the unit circle. We map $\mathbb{D}$ to the upper half space $\mathbb{H}^2$ by an isometry, such that the horocycle maps to the line $\Im z=1$. The image of $\odot P,\odot Q,\odot R$ are three circles orthogonal to $\Im z=1$, then their centers are collinear. Now it is worth noting that $\widehat{PQ}$, $\widehat{PR}$ and $\widehat{PR}$ are the Euclidean distance of these centers of circle. That is because inversive distance is invariant by inversion, and the hyperbolic distance between $\sqrt{-1}$ and $\tanh t+\frac{\sqrt{-1}}{\cosh t}$ on $\mathbb{H}^2$ is $t$. So there must be one formula among $\widehat{PQ}+\widehat{PR}=\widehat{QR}$, $\widehat{PQ}+\widehat{QR}=\widehat{PR}$ and $\widehat{PR}+\widehat{QR}=\widehat{PQ}$ that is true. By Formula \eqref{xi_trig}, we know $\Xi_{ijk}=0$, which is contradicted to $\Xi_{ijk}>0$.

    Suppose that the orthogonal circle is a non-compact hypercycle, or intersect with the unit circle at two points. On the geodesic connecting these two (infinite) points, there exist three points, denoted by $O_i$, $O_j$ and $O_k$, such that the geodesic and the orthogonal circle are all perpendicular to the geodesic connecting $O_\alpha$ and $v_\alpha$, where $\alpha=i,j,k$. The distances between the geodesic and the intersection points of the orthogonal circle and the vertex circles are equal, denoted by $\rho'>0$. Then by Formula \ref{cosh_ratio} we have 
    \[
        \sinh d(O_i,v_i)=p \sinh \rho' \quad \sinh d(O_j,v_j)=q \sinh \rho'
        \quad \sinh d(O_k,v_k)=r \sinh \rho'.
    \]    
    Let $x'=\cosh d(O_j,O_k)$, $y'=\cosh d(O_i,O_k)$ and $z'=\cosh d(O_i,O_j)$, by Formula \eqref{cosh_4} we have
    \begin{equation}\label{hypercycle}
    \begin{aligned}
        x'&=\frac{x+qr\sinh^2\rho'}{\sqrt{q^2\sinh^2\rho'+1}\sqrt{r^2\sinh^2\rho'+1}}\\
        y'&=\frac{y+pr\sinh^2\rho'}{\sqrt{p^2\sinh^2\rho'+1}\sqrt{r^2\sinh^2\rho'+1}}\\
        z'&=\frac{z+pq\sinh^2\rho'}{\sqrt{p^2\sinh^2\rho'+1}\sqrt{q^2\sinh^2\rho'+1}}.
    \end{aligned}        
    \end{equation}

    Since $O_i$, $O_j$ and $O_k$ locate on the same geodesic, thereby $2x'y'z'+1-x'^2-y'^2-z'^2=0$. Substitute \eqref{hypercycle} in it, we can solve that
    \begin{equation}%
        \frac{x^2+y^2+z^2-2xyz-1}{\Xi_{ijk}}=\sinh^2\rho'>0
    \end{equation}%
    Then the numerator is positive, which is contradicted to Formula \eqref{xi_delta}.

    Obviously the orthogonal circle cannot be outside from the unit dist.

    To sum up, the face $f_{ijk}$ has a compact orthogonal circle, that is why we call $\Xi_{ijk}$ the discriminant of compactness.
\end{proof}

\subsection{Weighted Voronoi inequality}

Secondly, we consider the inversive distance on a hinge $\Diamond_{ij;kl}$ embedded in $\mathbb{E}^2$. In this subsection, in case of too many subscripts, we denote $P=v_k,Q=v_i,R=v_l,S=v_j$ for vertices, $\odot P,\odot Q,\odot R,\odot S$ for circles, and $a=I_{ki},b=I_{il},c=I_{lj},d=I_{jk},e=I_{ij}$ for inversive distances. We denote that
\[
    p=\cosh r_k \quad q=\cosh r_i \quad r=\cosh r_l \quad s=\cosh r_j
\]
and
\[
    \hat{p}=\tanh r_k \quad \hat{q}=\tanh r_i \quad 
    \hat{r}=\tanh r_l \quad \hat{s}=\tanh r_j.
\]
If there is only one edge between two vertices without confusion, like $e_{ij}$ between $P$ and $Q$, then denote its hyperbolic length by $PQ$, etc. We denote the hyperbolic cosine of edge lengths as 
\begin{equation}\label{uvwxy}
    u=\cosh PQ,v=\cosh QR,w=\cosh RS,x=\cosh SP,y=\cosh QS.
\end{equation}

\begin{defn}\label{dht}
    Let $f$ satisfy the formula \eqref{f}, the \emph{local weighted Delaunay inequality} for the edge $e_{ij}$ is defined (in hyperbolic case) by
	\begin{equation}\label{ineq_h}
		\frac{\sqrt{\Delta _{bce}}}{\hat{p}} +
		\frac{\sqrt{\Delta _{ade}}}{\hat{r}}\le
		\frac{\sqrt{\Delta _{cdf}}}{\hat{q}} +
		\frac{\sqrt{\Delta _{abf}}}{\hat{s}}
	\end{equation}

    The set of $(\mathbf{I},\mathbf{r})$ satisfied the inequality for edge $e_{ij}$ of the triangulation $\mathcal{T}$, denote by $D_h(e_{ij};\mathcal{T})$, is a closed set in $\mathbb{R}_{>1}^{E(\mathcal{T})} \times \mathbb{R}_{>0}^V$. 
    
    The intersection of all these $\left|E\right|$ number of closed sets with respect to $\mathcal{T}$ is denoted by
    \begin{equation}\label{ineq_h2}
        D_h(\mathcal{T}) \coloneqq \bigcap_{e_{ij} \in E(\mathcal{T})}
        D_h(e_{ij};\mathcal{T}) \subset \mathbb{R}_{>1}^{E(\mathcal{T})}
        \times \mathbb{R}_{>0}^V.
    \end{equation}
\end{defn}

\begin{thm}\label{del_in_trig_h}
    Given a hyperbolic polyhedral surface $(S,V,d_h)$ with a weight function $r \in R(d_h)$, then the weight Delaunay triangulation $\mathcal{T}$ satisfies that every orthogonal circle on any face is compact. Namely, $D_h(\mathcal{T}) \subset \Xi(\mathcal{T})$.
\end{thm}

\begin{proof}
    Define an \emph{auxiliary length} $\hat{\mathbf{l}} \in \mathbb{R}_{>0}^{E(\mathcal{T})}$ by
    \begin{equation}%
        \hat{\mathbf{l}} \colon E(\mathcal{T}) \to
        \mathbb{R}_{>0} \quad e_{ij} \mapsto
        \sqrt{\tanh^2r_i+\tanh^2r_j+2I_{ij}\tanh r_i \tanh r_j}.
    \end{equation}%
    By notations, for $P=v_k,Q=v_i$, the auxiliary length $\widehat{PQ}=\hat{\mathbf{l}}(e_{ki})=\sqrt{\hat{p}^2+\hat{q}^2+2a\hat{p}\hat{q}}$ and so on. By Formula \eqref{xi_trig}, we know that the compactness orthogonal circle is equivalent to the triangle inequalities for auxiliary length. We claimed that when $\widehat{QR}<\widehat{QS}+\widehat{SR}$, $\widehat{SR}<\widehat{QS}+\widehat{QR}$ and the local weighted Delaunay inequality \eqref{ineq_h} hold, we have $\widehat{PQ}+\widehat{PS}>\widehat{QS}$.

    Suppose that $\widehat{PQ}+\widehat{PS} \le \widehat{QS}$, let
    \begin{equation}\label{G_p_infty}
        \hat{G}(\hat{p})\coloneqq \hat{p}^2+(a\hat{q}+d\hat{s})\hat{p}-e\hat{q}\hat{s}
        +\sqrt{\hat{p}^2+2a\hat{p}\hat{q}+\hat{q}^2}
        \sqrt{\hat{p}^2+2d\hat{p}\hat{s}+\hat{s}^2}.
    \end{equation}

    Let
    \begin{equation}%
        \hat{p}_0\coloneqq\frac{(e^2-1)\hat{q}\hat{s}}
        {\sqrt{\hat{q}^2+\hat{s}^2+2e\hat{q}\hat{s}}\sqrt{\Delta_{ade}}
        +(ae+d)\hat{q}+(de+a)\hat{s}}.
    \end{equation}%
    Similar to Lemma 4.3 in \cite{zhu2023existence}, we can verify that $\hat{p}_0$ is the unique zero point of $\hat{G}$ on $\mathbb{R}_{>0}$.

    If $\hat{p}_0 \ge 1$, then we have $\tanh r_k<\hat{p}_0$, which is trivial since $\tanh r_k<1$. 
    
    If $0<\hat{p}_0<1$, Since
    \begin{equation}\label{G_p_0_1}
        \hat{G}(\hat{p})=\frac12(\widehat{PQ}+\widehat{PS}-\widehat{QS})(\widehat{PQ}+\widehat{PS}+\widehat{QS})
    \end{equation}
    holds when $0<\hat{p}_0<1$, by monotonically increasing of $\hat{G}$ and $\widehat{PQ}+\widehat{PS} \le \widehat{QS}$, we also have $0<\hat{p}<\hat{p}_0$.

    Next we discuss case by case. 

    For the first case, if $(b^2-1)\hat{q}^2-2(bc+e)\hat{q}\hat{s}+(c^2-1)\hat{s}^2 \le 0$, we can verify
    \[
        \widehat{QS}\sqrt{\Delta_{bce}} \ge be\hat{q}+c\hat{q}+ce\hat{s}+b\hat{s}>0
    \]
    by taking square power, then
    \[
    \begin{aligned}
        &\frac{\sqrt{\Delta_{bce}}}{\hat{p}}+\frac{\sqrt{\Delta_{ade}}}{\hat{r}}-
        \left(\frac{\sqrt{\Delta_{cdf}}}{\hat{q}}+
        \frac{\sqrt {\Delta _{abf}}}{\hat{s}}\right)\\
        >& \frac{\sqrt{\Delta _{bce}}}{\hat{p}_0}-
        \left(\frac{\sqrt{\Delta_{cdf}}}{\hat{q}}+\frac{\sqrt{\Delta_{abf}}}{\hat{s}}\right)\ge 0,
    \end{aligned}
    \]
    which is contradicted to \eqref{ineq_h}.
    
    For the second case, if $(b^2-1)\hat{q}^2-2(bc+e)\hat{q}\hat{s}+(c^2-1)\hat{s}^2 > 0$, then $\widehat{QS}\sqrt{\Delta_{bce}}<be\hat{q}+c\hat{q}+ce\hat{s}+b\hat{s}$. Let
    \[
        \hat{r}_0 \coloneqq \frac{(e^2-1)\hat{q}\hat{s}}{(be+c)\hat{q}+(ce+b)\hat{s}-
        \sqrt{\hat{q}^2+\hat{s}^2+2e\hat{q}\hat{s}}\sqrt{\Delta_{bce}}}.
    \]
    Similarly, we can conclude that $0<\hat{r}<\hat{r}_0$, then
    \[
        \frac{\sqrt {\Delta _{bce}}}{\hat{p}} + \frac{\sqrt {\Delta _{ade}}}{\hat{r}}>
        \frac{\sqrt {\Delta _{bce}}}{\hat{p}_0} + \frac{\sqrt {\Delta _{ade}}}{\hat{r}_0}=
        \frac{\sqrt {\Delta _{cdf}}}{\hat{q}} + \frac{\sqrt {\Delta _{abf}}}{\hat{s}}
    \]
    which is also contradicted to \eqref{ineq_h}.

    To sum up, we have $\widehat{PQ}+\widehat{PS}>\widehat{QS}$.

    If there is $(\mathbf{I},\mathbf{r}) \in D_h(\mathcal{T})$ such that for the inversive distance circle packing $\tilde{L}_h(\mathbf{I},\mathbf{r})$, there exist some faces of which the orthogonal circle is non-compact. By taking the edge with the largest auxiliary length among these faces, we can easily find the contradiction by the same reason of Theorem 4.4 in \cite{zhu2023existence}. Thus, $D_h(\mathcal{T}) \subset \Xi(\mathcal{T})$.
\end{proof}

\subsection{Weighted Delaunay condition}

\begin{lem}
    With the notations above, when the orthogonal circles of two faces of the hinge $\Diamond_{ij;kl}$ are compact, the edge $e_{ij}$ is local weighted Delaunay if and only if the inequality \eqref{ineq_h} holds, where $f$ satisfies \eqref{f}.
\end{lem}

\begin{proof}    
    Substitute \eqref{delta} into \eqref{ineq_h} and simplify, we have
    \begin{equation}%
        \begin{aligned}
            &\hat{p}((be+c)\hat{q}\hat{r}+(ce+b)\hat{r}\hat{s}-(e^2-1)
            \hat{q}\hat{s})\sqrt{\Delta_{ade}}\\
            +&\hat{r}((ae+d)\hat{p}\hat{q}+(de+a)\hat{p}\hat{s}-(e^2-1)
            \hat{q}\hat{s})\sqrt{\Delta_{bce}} \ge 0.
        \end{aligned}
    \end{equation}%

    Recall the notations in \ref{uvwxy}, we have
    \begin{equation}%		
        \begin{aligned}
            &(ae+d)\hat{p}\hat{q}+(de+a)\hat{p}\hat{s}-(e^2-1)\hat{q}\hat{s})\\
            =&\frac{(xy-u)q+(uy-x)s-(y^2-1)p}{pqs\sqrt{q^2-1}\sqrt{s^2-1}}\\
            &(be+c)\hat{q}\hat{r}+(ce+b)\hat{r}\hat{s}-(e^2-1)\hat{q}\hat{s})\\
            =&\frac{(wy-v)q+(vy-w)s-(y^2-1)r}{qrs\sqrt{q^2-1}\sqrt{s^2-1}}.
        \end{aligned}
    \end{equation}%

    In \cite{zhang2014unified}, we already know that
    \begin{equation}%
        (y^2-1)\sinh^2 h_{ij,k}=(2qsy-q^2-s^2)\cosh^2 \rho_{ijk}+1-y^2.
    \end{equation}%
    Keep computing, then
    \begin{equation}%
        (y^2-1)\sinh^2 h_{ij,k}=\frac{((xy-u)q+(uy-x)s-(y^2-1)p)^2}{\Xi_{ijk}}.
    \end{equation}%

    Take square root of both sides, and determine the sign of $h_{ij,k}$, we have
    \begin{equation}%
        \sinh h_{ij,k}\sqrt{(y^2-1)\Xi_{ijk}}=(xy-u)q+(uy-x)s-(y^2-1)p
    \end{equation}%
    and
    \begin{equation}%
        \sinh h_{ij,l}\sqrt{(y^2-1)\Xi_{ijl}}=(wy-v)q+(vy-w)s-(y^2-1)r
    \end{equation}%
    similarly.

    So \eqref{ineq_h} is equivalent to
    \begin{equation}%
        \hat{p}\frac{\sinh h_{ij,l}\sqrt{(y^2-1)\Xi_{ijl}}}
        {qrs\sqrt{q^2-1}\sqrt{s^2-1}}\sqrt{\Delta_{ade}}
        +\hat{r}\frac{\sinh h_{ij,k}\sqrt{(y^2-1)\Xi_{ijk}}}
        {pqs\sqrt{q^2-1}\sqrt{s^2-1}}\sqrt{\Delta_{bce}} \ge 0
    \end{equation}%
    which is
    \begin{equation}%
        \frac{\sinh h_{ij,k}\sqrt{\Xi_{ijk}}}{\sqrt{p^2-1}\sqrt{\Delta_{ade}}} +
        \frac{\sinh h_{ij,l}\sqrt{\Xi_{ijl}}}{\sqrt{r^2-1}\sqrt{\Delta_{bce}}}
        \ge 0
    \end{equation}%

    By \eqref{radius_tanh}, we have
    \begin{equation}%
        \frac{\sinh h_{ij,k}}{\sinh \rho_{ijk}} +
        \frac{\sinh h_{ij,l}}{\sinh \rho_{ijl}}\ge 0.
    \end{equation}%

    Since the geodesic connecting $O_{ijk}$ and $O_{ijl}$ is perpendicular to $e_{ij}$ and intersect on $e_{ij}$, by \eqref{cos} we know
    \begin{equation}%
        \frac{\cosh h_{ij,k}}{\cosh \rho_{ijk}} =
        \frac{\cosh h_{ij,l}}{\cosh \rho_{ijl}} <1.
    \end{equation}%

    Therefore, $\frac{\sinh h_{ij,k}}{\sinh \rho_{ijk}}$ is a strictly increasing odd function of $h_{ij,k}$. Thus, the inequality \eqref{ineq_h} is equivalent to $h_{ij,k}+h_{ij,l} \ge 0$
\end{proof}

Similar to Theorem 4.6 in \cite{zhu2023existence}, using the above lemma, the following theorem can be generalized without modification. In the proof we need the existence of compact orthogonal circle of faces on $\mathcal{T}$, which is guaranteed by Theorem \ref{del_in_trig_h}.

\begin{thm}\label{whole_del_h}
    Given a hyperbolic polyhedral surface $(S,V,d_h)$ with a weight function $r \in R(d_h)$, the triangulation $\mathcal{T}$ is weighted Delaunay if and only if there exist $(\mathbf{I},\mathbf{r}) \in D_h(\mathcal{T})$ such that $\Psi_\mathcal{T} \circ L_h(\mathbf{I},\mathbf{r})=[d_h]$.
\end{thm}

\section{Diffeomorphism between Teichm\"uller spaces}\label{section4}

\subsection{Homeomorphism between Teichm\"uller spaces}

Recall that given a marked surface $(S,V)$, the Teichm\"uller space of hyperbolic polyhedral metric is
\[
    Teich_h(S,V)=\bigcup_\mathcal{T}\Psi_\mathcal{T}
    (\mathbb{R}^{E(\mathcal{T})}_\Delta).
\]
Then we attach the weight $R$ with respect to $R(d_h)$ Teichm\"uller spaces. Define that
\[
    \widetilde{Teich}_h(S,V) \coloneqq \left\{([d_h],\mathbf{r})
    \mid [d_h]\in Teich_h(S,V), \mathbf{r} \in R(d_h) \right\}
\]
and the map
\begin{equation}% 
    \tilde{\Psi}_\mathcal{T} \colon Q_h(\mathcal{T}) \subset
    \mathbb{R}_{>1}^{E(\mathcal{T})} \times \mathbb{R}_{>0}^V \to
    Teich_h(S,V) \times \mathbb{R}_{>0}^V \quad (\mathbf{I},\mathbf{r}) 
    \mapsto (\Psi_\mathcal{T} \circ L_h(\mathbf{I},\mathbf{r}),\mathbf{r}).
\end{equation}%
Moreover,
\begin{equation}% 
    \tilde P_h(\mathcal{T}) \coloneqq \tilde\Psi_\mathcal{T}(Q_h(\mathcal{T})).
\end{equation}%

Since $\Psi_\mathcal{T} \colon \mathbb{R}^{E(\mathcal{T})}_\Delta \to P_h(\mathcal{T})$ is a homeomorphism, by Proposition \ref{lift_inject}, we know $\tilde\Psi_\mathcal{T}\colon Q_h(\mathcal{T}) \to \tilde P_h(\mathcal{T})$ is a real analytic homeomorphism.

Let $\pi([d_h],\mathbf{r})=[d_f]$. Since $L_h(Q_h\mathcal{T})=\mathbb{R}^{E(\mathcal{T})}_\Delta$, we know $\pi(\tilde P_h(\mathcal{T}))=\Psi_\mathcal{T} \circ L_h(Q_h(\mathcal{T}))=P_h(\mathcal{T})$.

Recall the Teichm\"uller space of hyperbolic surface with geodesic boundaries as well. Suppose a surface $\Sigma$ with $\chi(\Sigma)<0$ and a marked surface $(S,V)$ are related, given a triangulation $\mathcal{T}=(V,E,F)$, for any length coordinate $\mathbf{x} \in \mathbb{R}_{>0}^{E(\mathcal{T})}$, there is an equivalent class up to an isometry isotopic to the identity on $\Sigma$ with a representative element $d$, which is $[d]=\Omega_\mathcal{T}(\mathbf{x}) \in Teich(\Sigma)$.

To attach the weight on the Teichm\"uller space, we define
\begin{equation}\label{omega_h}
    \tilde{\Omega}_\mathcal{T} \colon \mathbb{R}_{>1}^{E(\mathcal{T})}
    \times \mathbb{R}_{>0}^V \to Teich(\Sigma) \times 
    (0,1)^V \quad (\mathbf{I},\mathbf{r}) \mapsto
    (\Omega_\mathcal{T}\circ \arcch(\mathbf{I}),\tanh(\mathbf{r})),
\end{equation}
which is a homeomorphic map by Theorem \ref{teich home formula}.

With these notations, by Theorem \ref{del_in_trig_h} and Theorem \ref{whole_del_h} we know that.
\[
    \tilde\Psi_\mathcal{T}(D_h(\mathcal{T}))\subset\widetilde{Teich}_h(S,V)
\]

Considering Theorem \ref{local_whole_h} and Theorem \ref{gbh_unique}, we have
\begin{equation}\label{union_teich_h}
    \begin{aligned}
        \widetilde{Teich}_h(S,V)&=\bigcup_\mathcal{T}
        \tilde\Psi_\mathcal{T}(D_h(\mathcal{T}))
        \\
        Teich(\Sigma) \times (0,1)^V&=\bigcup_\mathcal{T}
        \tilde\Omega_\mathcal{T}(D_h(\mathcal{T})).
    \end{aligned}
\end{equation}

Now we define
\begin{equation}\label{B_T}
    B_\mathcal{T} \coloneqq \tilde\Omega_\mathcal{T} \circ 
    \tilde\Psi_\mathcal{T}^{-1} \colon \tilde P_h(\mathcal{T})
    \to Teich(\Sigma) \times (0,1)^V,
\end{equation}
which is a composition of two real analytic homeomorphism. Since the image of $\tilde\Psi_\mathcal{T}^{-1}$ is a subset of the domain of $\tilde\Omega_\mathcal{T}$, we know $B_\mathcal{T}$ is a real analytic injective map.

For two different triangulations $\mathcal{T}$ and $\mathcal{T}'$, if $P_h(\mathcal{T})\cup P_h(\mathcal{T}') \ne \varnothing$, let $([d_h],\mathbf{r})\in\tilde P_h(\mathcal{T})\cup P_h(\mathcal{T}')$, generally, $B_\mathcal{T}([d_h],\mathbf{r}) \ne B_{\mathcal{T}'}([d_h],\mathbf{r})$. However, we have the following lemma:
\begin{lem}\label{intersection}
    Given two triangulations $\mathcal{T}$ and $\mathcal{T}'$ of the surface $(S,V)$ such that $\tilde\Psi_\mathcal{T}(D_h(\mathcal{T})) \cap \tilde\Psi_{\mathcal{T}'}(D_h(\mathcal{T}')) \ne \varnothing$, suppose $([d_h],\mathbf{r}) \in \tilde\Psi_\mathcal{T}(D_h(\mathcal{T})) \cap \tilde\Psi_{\mathcal{T}'}(D_h(\mathcal{T}'))$, then
    \[
        B_\mathcal{T}([d_h],\mathbf{r})=B_{\mathcal{T}'}([d_h],\mathbf{r}).
    \]
\end{lem}

\begin{proof}
    By Theorem \ref{whole_del_h} we know that $\mathcal{T}$ and $\mathcal{T}'$ are all weight Delaunay triangulation of $(S,V,d_h)$ with the weight $\mathbf{r}$. By Theorem \ref{local_whole_h} we know that they differ by a finite number of edge switches.

    Without loss of generality, we suppose that $\mathcal{T}$ and $\mathcal{T}'$ differ by one switch of the hinge $\Diamond_{ij;kl}$. 

    With the notations in Formula \eqref{ineq_h}, the components of $\Psi_{\mathcal{T}'}^{-1}([d],\mathbf{r})$ at $e_{ki}$, $e_{il}$, $e_{lj}$, $e_{jk}$, $v_k$, $v_i$, $v_l$ and $v_j$ are equal to the ones of $\Psi_\mathcal{T}^{-1}([d],\mathbf{r})$, while the inversive distance at $e_{kl}$ is denoted by $f$. By Theorem \ref{teich home formula}, to prove $B_\mathcal{T}([d_h],\mathbf{r})=B_{\mathcal{T}'}([d_h],\mathbf{r})$, we only need to show that $f$ satisfies Formula \eqref{f}.

    The hinge $\Diamond_{ij;kl}$ being able to switch means $O_{ijk}$ and $O_{ijl}$ coincide. After developing them onto $\mathbb{D} \cong \mathbb{H}^2$, the four circles at vertices share a common orthogonal circle, which is compact by Theorem \ref{del_in_trig_h}.

    The automorphism of $\mathbb{D}$ is an isometry of hyperbolic plane, and also a conformal map of Euclidean plane. By an automorphism, we can construct the map such that the Euclidean center of the orthogonal circle locate at the center of $\mathbb{D}$. From \cite{guo2011local} we know the inversive distances are invariant under $\mbox{Aut}(\mathbb{D})$. The orthogonal circle are invariant change by switching, so are the radii of them. By this assumption, the image of isometries from $\Diamond_{ij;kl}$ and $\Diamond_{ij;kl}$ are rigid and differ by a rotation. By definition of inversive distance on six edges between four circles, we know that $f$ satisfies Formula \eqref{f}.

    For the case that $\mathcal{T}$ and $\mathcal{T}'$ differ by a finite number of switches, just repeat the step above. The image of $([d_f],\mathbf{r})$ does not change under different mappings.
\end{proof}

We can glue all the mappings $B_\mathcal{T}$ together, and construct the mapping between two different kinds of Teichm\"uller space of $(S,V)$. 

\begin{lem}\label{A_home}
    The glued mapping
    \begin{equation}%
        \mathbf{B} \coloneqq \bigcup_\mathcal{T}B_\mathcal{T}|_
        {\tilde\Phi_\mathcal{T}(D_h(\mathcal{T}))} \colon \widetilde{Teich}_h(S,V)
        \to Teich(\Sigma) \times (0,1)^V
    \end{equation}%
    is well-defined and homeomorphic.
\end{lem}

\begin{proof}
    We know $\mathbf{B}$ is well-defined by Lemma \ref{intersection} and Equation \eqref{union_teich_h}, 

    % Since $\mathbf{B}$ is a piecewise real analytic map such that the image at the intersections of the pieces are the same, we know that $\mathbf{B}$ is continuous. To prove that $\mathbf{B}$ is homeomorphic, we directly construct the inverse map of $\mathbf{B}$ and prove that the inverse map is continuous.

    Firstly, we claim that $\mathbf{B}$ is injective. 

    Given any $([d_h],\mathbf{r}) \ne ([d_h'],\mathbf{r}') \in \widetilde{Teich}_h(S,V)$, if their weighted Delaunay triangulation are isotopic, namely there exist $\mathcal{T}$ such that $([d_h],\mathbf{r}),([d_h'],\mathbf{r}') \in \tilde\Psi_\mathcal{T}(D_h(\mathcal{T}))$, by Lemma \ref{lift_inject}, or $\tilde{\Psi}_\mathcal{T}$ and $\tilde{\Omega}_\mathcal{T}$ being homeomorphic, we have $\mathbf{B}([d_h],\mathbf{r}) \ne \mathbf{B}([d_h'],\mathbf{r}')$. 

    If any of their weighted Delaunay triangulation are not isotopic, namely there exist two different triangulation $\mathcal{T}$ and $\mathcal{T}'$ such that $([d_h],\mathbf{r}) \in \tilde\Psi_\mathcal{T}(\interior D_h(\mathcal{T}))$ and $([d_h'],\mathbf{r}') \in \tilde\Psi_\mathcal{T}(\interior D_h(\mathcal{T}'))$, then $\mathbf{B}$ map them to different $n+\left|E\right|$-cells of the cell decomposition.

    We already know that \ref{union_teich_gb} is a cell decomposition. Restricted on $Teich(\Sigma) \times (0,1)^V$, it is also a cell decomposition. Besides, from \eqref{ineq_f} and \eqref{ineq_h}, we know $D_h(\mathcal{T})$ is real analytic homeomorphic to $D_h(\mathcal{T})$ by the mapping 
    \[
        \left(id_{\mathbb{R}_{>0}^{E(\mathcal{T})}},\arcth\right) \colon
        \mathbb{R}_{>0}^{E(\mathcal{T})} \times (0,1)^V \to
        \mathbb{R}_{>0}^{E(\mathcal{T})} \times \mathbb{R}_{>0}^V.
    \]
    Moreover, by definition of $\mathbf{A}$ in \cite{zhu2023existence},
    \begin{equation}%
        \mathbf{B}(D_h(\mathcal{T}))=\mathbf{A}(D_f(\mathcal{T}))\cap
        Teich(\Sigma) \times (0,1)^V.
    \end{equation}%

    Secondly, by inequalities \eqref{ineq_f} and \eqref{ineq_h}, the weighted Delaunay triangulation of $(S,V,d_h)$ with weight $\mathbf{r} \in R(d_h)$, is isotopic to the truncated weighted Delaunay triangulation of $(\Sigma,d)$ with the weight $\tanh\mathbf{r}$, or differed by some edge switch, where $d$ satisfies $([d],\mathbf{r})=\mathbf{B}([d_h],\mathbf{r})$. Denote the common triangulation by $\mathcal{T}$. Therefore, for any $([d],\mathbf{r}) \in Teich(\Sigma) \times(0,1)^V$, we can construct its preimage $\tilde{\Psi}_\mathcal{T}\circ\tilde{\Omega}_\mathcal{T}^{-1} ([d],\arcth(\mathbf{r}))$, so $\mathbf{B}$ is a surjection and furthermore a bijection.

    Finally, when $([d],\mathbf{r})$ locate on a lower dimensional cell than $n+\left|E\right|$, the ability to switch on both sides of the common weighted Delaunay triangulation are equivalent. By Lemma \ref{intersection}, the preimages of $([d],\mathbf{r})$ under $B_\mathcal{T}$ with different possible triangulations are equal. Therefore, $\mathbf{B}^{-1}$ is continuous. To sum up, $\mathbf{B}$ is homeomorphic.
\end{proof}

The proof above also claimed that
\begin{prop}\label{cell}
    Equation \eqref{union_teich_h} forms a local finite cell decomposition.
\end{prop}

\subsection{First derivative of the diffeomorphism}

\begin{lem}\label{dfdF}
    Let the hinge $\Diamond_{ij;kl}$ be assigned with radii $r_k,r_i,r_l,r_j>0$ at vertices $v_k,v_i,v_l,v_j$, and with inversive distances $a,b,c,d,e>1$ at edges $e_{ki}$, $e_{il}$, $e_{lj}$, $e_{jk}$ and $e_{ij}$. On the hyperbolic polyhedral metric computed by the inversive distance circle packing $L_h$, the \emph{hyperbolic cosines} of corresponding hyperbolic edge lengths are denoted as $u,v,w,x,y$. The hyperbolic cosines of radii are denoted as $p,q,r,s$ in corresponding order.

    If the faces $f_{ijk}, f_{ijl}$ have compact orthogonal circles, we develop and embed $\Diamond_{ij;kl}$ into $\mathbb{H}^2$, and let $z = d_{\mathbb{H}^2}(v_k,v_l)$.
    
    Let $F=\frac{z-pr}{\sqrt{p^2-1}\sqrt{r^2-1}}$, and let $f$ satisfy Equation \eqref{f}. Consider $F$ and $f$ as functions of variables $p,q,r,s,a,b,c,d,e$. When the local weighted Delaunay inequality \eqref{ineq_f} holds with equality, we have
    \[
        dF=df.
    \]
\end{lem}

\begin{proof}
    In the developed hinge $\Diamond_{ij;kl}$, denoted the inner angle of the face $f_{ijk}$ at the vertex $v_i$ by $\alpha$, and the one of the face $f_{ijl}$ by $\beta$. By cosine law \eqref{cos},
    \[
        \cos \alpha = \frac{uy-x}{\sqrt{u^2-1}\sqrt{y^2-1}}\quad
        \cos \beta = \frac{vy-w}{\sqrt{v^2-1}\sqrt{y^2-1}} \quad
        \cos(\alpha+\beta)=\frac{uv-z}{\sqrt{u^2-1}\sqrt{v^2-1}}.
    \]
    Substitute them in the identity $(\cos(\alpha+\beta)-\cos\alpha\cos\beta)^2=(1-\cos^2\alpha)(1-\cos^2\beta)$ and simplify, we have
    \[        
        \begin{aligned}
            0=&u^2 w^2+v^2 x^2+y^2 z^2-u^2-v^2-w^2-x^2-y^2-z^2+1\\
            -&2(u v w x+u w y z+v x y z-v w y-u x y-u v z-w x z).
        \end{aligned}
    \]

    Substituting the formulas for the hyperbolic cosine of edge length, i.e., $u=pq+a\sqrt{p^2-1}\sqrt{q^2-1}$, and the special one $z=pr+F\sqrt{p^2-1}\sqrt{r^2-1}$, into the above equation, we have
    \[
        X(p,q,r,s,a,b,c,d,e,F)=0,
    \]
    where $X$ is a ten-variable polynomial of degree twelve with each variable having a maximum degree of two. The specific form is too lengthy to be included here. Denote the generalized Ptolemy equation \eqref{abcdef} by $Y(a,b,c,d,e,f)=0$.   

    From here we use a different method from \cite{zhu2023existence}. Note that $X$ is a quadratic function of $p$, so there exists a unique critical point $p=p_1$ such that $X_p(p_1,q,r,s,a,b,c,e,d,F)=0$. 
    
    Substituting $p=p_1(q,r,s,a,b,c,d,e,F)$ into $X(p,q,r,s,a,b,c,d,e,F)$, then there is a factor $Y(a,b,c,d,e,F)$ in the numerator of $X(q,r,s,a,b,c,d,e,F)$ that we get, and other factors in the numerator are obviously non-zero. 
    
    Thus, if $X(p_1,q,r,s,a,b,c,e,d,F)=0$, then $Y(a,b,c,d,e,F)=0$. That means when $p=p_1$, we have $F=f$ by the geometry meaning of hinge $\Diamond_{ij;kl}$, then the four vertex circles are orthogonal to a common circle, and $\Diamond_{ij;kl}$ is able to switch.

    By the formula for implicit derivative $F_p=-X_p/X_F$, after substituting $F=f$ and $p=p_1$, we have $F_p=0=f_p$.

    Substituting $F=f$ and $p=p_1$ into $X_q$, $Y_a/Y_f-X_a/X_F$ and $Y_e/Y_f-X_e/X_F$, respectively, and combining them by finding a common denominator and factoring, we see that all the numerators contain the factor $Y(a,b,c,d,e,f)$. Since $f$ satisfies $Y(a,b,c,d,e,f)=0$, the three results are equal to $0$. Then we have $F_q=-X_q/X_F=0$, $Y_a/Y_f-X_a/X_F=0$ and $Y_e/Y_f-X_e/X_F=0$, that is to say $F_q=f_q=0$,$F_a=f_a$ and $F_e=f_e$.

    Using the symmetry of the variables in the expressions for $X$ and $Y$, we can derive other equalities between partial derivatives: $F_r=f_r=F_s=f_s=0$, and $F_b=f_b,F_c=f_c,F_d=f_d$. Thus, this lemma $dF=df$ is proved. See codes in appendix.
\end{proof}

\begin{thm}\label{c1diff}
    $\mathbf{B}$ is a $C^1$ diffeomorphism.
\end{thm}

\begin{proof}
    For any $([d_h],\mathbf{r}) \in \widetilde{Teich}_f(S,V)$, if there exists a triangulation $\mathcal{T}$ such that $([d_h],\mathbf{r}) \in \interior \tilde\Omega_\mathcal{T}(D_f(\mathcal{T}))$, we know that $\mathbf{B}$ is smooth near $\tilde\Omega_\mathcal{T}([d_h],\mathbf{r})$ because $A_\mathcal{T}$ is real analytic.

    If not, there still exists a triangulation $\mathcal{T}$ such that $([d_h],\mathbf{r}) \in \tilde\Omega_\mathcal{T}(D_f(\mathcal{T}))$. In this case, since $\tilde\Omega_\mathcal{T}$ is homeomorphic, we have $([d_h],\mathbf{r}) \in \partial\tilde\Omega_\mathcal{T}(D_f(\mathcal{T})) = \tilde\Omega_\mathcal{T}(\partial D_f(\mathcal{T}))$. 
    By the real analytic cell decomposition structure in Proposition \ref{cell}, there exist a finite number of triangulations denoted by $\mathcal{T}_1=\mathcal{T},\dots,\mathcal{T}_k,k \ge 2$, such that they differ from $\mathcal{T}$ by a finite number of edge switches, which means $([d_h],\mathbf{r}) \in \tilde\Omega_\mathcal{T}(\bigcap_{i=1}^k D_f(\mathcal{T}_i))$.
    Moreover, there exist a neighborhood $U$ such that $([d_h],\mathbf{r}) \in U \subset \tilde\Omega_\mathcal{T}(\bigcup_{i=1}^k D_f(\mathcal{T}_i))$. To prove that $\mathbf{A}$ is $C^1$ near $\tilde\Omega_\mathcal{T}([d_h],\mathbf{r})$, it is sufficient to show that $dB_{\mathcal{T}}=dB_{\mathcal{T}'}$ at $\tilde\Omega_\mathcal{T}([d_h],\mathbf{r})$ for any $\mathcal{T}' \in \left\{\mathcal{T}_2,\dots,\mathcal{T}_k\right\}$.
    
    Consider $\tilde\Omega_{\mathcal{T}'} \circ \tilde\Omega_\mathcal{T}^{-1}$ and $\tilde\Psi_{\mathcal{T}'} \circ \tilde\Psi_\mathcal{T}^{-1}$.
    Firstly, we show that $d(\tilde\Omega_{\mathcal{T}'} \circ \tilde\Omega_\mathcal{T}^{-1}) = d(\tilde\Psi_{\mathcal{T}'} \circ \tilde\Psi_\mathcal{T}^{-1})$ holds at $([d_h],\mathbf{r})$, if $\mathcal{T}$ and $\mathcal{T}'$ are differed by one edge switch.
    By observing the coordinates of these two maps, we see that they are identity on every coordinate except the switched one. The two different inversive distance values (i.e., $f$ and $F$ in Lemma \ref{dfdF}) have the same first-order partial derivatives with respect to other coordinates by Lemma \ref{dfdF}.
    Thus, the two maps $\tilde\Omega_{\mathcal{T}'} \circ \tilde\Omega_\mathcal{T}^{-1}$ and $\tilde\Psi_{\mathcal{T}'} \circ \tilde\Psi_\mathcal{T}^{-1}$ have the same Jacobi matrices.
    Then, by $\tilde\Omega_{\mathcal{T}'}$ and $\tilde\Psi_\mathcal{T}$ are real analytic homeomorphic, we know that $dB_{\mathcal{T}}=dB_{\mathcal{T}'}$ holds at $\tilde\Omega_\mathcal{T}([d_h],\mathbf{r})$.
    For $\mathcal{T}$ and $\mathcal{T}'$ differed by finite switches, just repeat the steps above for finite times.
    
    Therefore, we have proved that $\mathbf{B}$ is a global $C^1$ diffeomorphism. It can be verified that $\mathbf{B}$ is not a $C^2$ diffeomorphism, but this is irrelevant and omitted.
\end{proof}

\section{Proof of the main theorem}\label{section5}

\subsection{Discrete conformal equivalence}
The discrete conformal equivalent class of inversive distance circle packing on polyhedral surface is defined as follows.

\begin{defn}
    Given a marked surface $(S,V)$, suppose $d_h$ and $d_h'$ are two hyperbolic polyhedral metrics with a legal weight function respectively $\mathbf{r} \in R(d_h)$ and $\mathbf{r}' \in R(d_h')$ respectively, we say $([d_h],\mathbf{r})$ and $([d_h'],\mathbf{r}')$ are \emph{discrete conformal equivalent} for inversive distance circle packing, if the first component of $\mathbf{B}([d_h],\mathbf{r})$ and $\mathbf{B}([d_h'],\mathbf{r}')$ are equal, which means two derived hyperbolic metrics with geodesic boundaries are isotopic.
\end{defn}

Since $\mathbf{B}$ is homeomorphic, this definition of equivalence has reflexivity, symmetry, and transitivity, thus it is well-defined. Then we have the following proposition.

\begin{prop}
    Given a marked surface $(S,V)$ and its related surface $\Sigma$, for any discrete conformal equivalent class, it can be represented by
    \[
        \mathbf{B}^{-1}\left(\left\{[d]\right\} \times (0,1)^V\right)
        \subset \widetilde{Teich}_h(S,V),
    \]
    where $[d]$ is an isometry class isotopic to identity on the hyperbolic surface with geodesic boundaries. 
\end{prop}

Here is the main theorem of this paper.

\begin{thm}\label{main_h}
    Given a marked surface $(S,V)$ with $n$ vertices, for any $([d_f],\mathbf{r}) \in \widetilde{Teich}_f(S,V)$ and target discrete curvature
    \[
        \bar{\mathbf K} \colon V \to (-\infty,2\pi) \quad v_i \mapsto \bar K_i
    \]
    satisfies the Gauss-Bonnet inequality formula
    \[
        \sum_{i=1}^n\bar K_i>2\pi \chi(S),
    \] 
    then there exist a unique inversive distance circle packing $([d_f'],\mathbf{r}') \in \widetilde{Teich}_f(S,V)$ discrete conformal equivalent to $([d_f],\mathbf{r})$, such that the discrete curvature of the piecewise flat metric $d_f'$ at the $v_i \in V$ is equal to $\bar K_i$. 
\end{thm}

\subsection{Variational principle}

The following Lemma and Theorem can be found in \cite{guo2011local} and \cite{zhang2014unified}.

\begin{lem}\label{diff_trig}
    Given a marked surface $(S,V)$ with a triangulation $\mathcal{T}=(V,E,F)$, for any $(\mathbf{I},\mathbf{r}) \in Q_h(\mathcal{T})$, construct the inversive distance circle packing $L_h(\mathbf{I},\mathbf{r})=\mathbf{l}$. Denote the inner angle of $f_{ijk}$ at vertices $v_i,v_j,v_k$ by $\theta_i,\theta_j,\theta_k$, and the opposite edge lengths by $l_i,l_j,l_k$. For any vertex $v_i$, let 
    \[
        u_i= \log \tanh \frac{r_i}{2} \in (-\infty,0) \quad \mbox{or} \quad 
        \tanh r_i=\frac{1}{\cosh u_i} \in (0,1)
    \]
    then the Jacobi matrix
    \[
        \frac{\partial(\theta_i,\theta_j,\theta_k)}{\partial(u_i,u_j,u_k)}
    \]
    is symmetric and negative definite.
\end{lem}

Since $K_i=2\pi-\sum_{v_i \in f} \theta_i$, by the lemma above, we have the following theorem.

\begin{thm}\cite{guo2011local}\cite{zhang2014unified}\label{diff_all}
    Given a piecewise flat surface $(S,V,d_f)$ and weight $\mathbf{r} \in R(d_f)$, denote the discrete curvature at vertices by $\mathbf{K}$, then the weighted Delaunay triangulation $\mathcal{T}$ satisfies $([d_h],\mathbf{r}) \in \tilde\Psi_\mathcal{T}(D_h(\mathcal{T}))$. Let $\mathbf{u}=\log\tanh(\frac12\mathbf{r})$, then the Jacobi matrix $\dfrac{\partial \mathbf{K}}{\partial \mathbf{u}}$ is symmetric and negative definite.
\end{thm}

The Jacobi matrix above is a sparse matrix formed by combining matrices from Lemma \ref{diff_trig}. If there is a face glued by itself in the triangulation, we can simply add the elements at the corresponding positions of the matrix, and the conclusion remains unchanged.

Lemma \ref{diff_trig} and Theorem \ref{diff_all} describe the local differential properties. With the help of Theorem \ref{c1diff}, we can define the curvature map globally.

\begin{defn}\label{kappa}
    Suppose $\Sigma$ is the related surface of $(S,V)$ with $n$ vertices. Given a hyperbolic metric with geodesic boundaries $d$ on $\Sigma$, define that  
    \[
        \kappa_d \colon \mathbb{R}_{<0}^n \to (-\infty,2\pi)^n \quad
        \mathbf{u} \mapsto \mathbf{K}=(K_1,\dots,K_n),
    \]
    where $\mathbf{K}$ is the discrete curvature of the metric at the first component in $\mathbf{B}^{-1}([d],\sech(\mathbf{u}))$. 
    
    Moreover, define the \emph{Ricci potential} as
    \[
        \mathcal{E}_d=\mathcal{E}_d(\mathbf{w}) \coloneqq \int ^\mathbf{w} \sum_{i=1}^n K_i\,du_i.
    \]
\end{defn}

Considering that $\mathbf{B}$ is $C^1$ and the map from edge length to curvature is real analytic by cosine law, we know $\kappa_d$ is $C^1$ as a restriction of the composition of these two maps. Note that the domain of $\mathbf{K}$ here is expanded than the one in Theorem \ref{diff_all}. To show that the Ricci potential is well-defined, let
\[
    U_i\coloneqq \left\{\, \mathbf{u}\in\mathbb{R}_{<0}^V \mid
    \tilde\Omega_{\mathcal{T}_i} \left([d],\arcth\circ\sech(\mathbf{u}) \right)
    \subset D_h(\mathcal{T}_i) \,\right\}.
\]
Zhu proved that the number of the set like this finite in \cite{zhu2019discrete}. Denote them by $U_1,\dots,U_M$, then
\begin{equation}\label{rn_cell}
    \mathbb{R}_{<0}^V=\bigcup_{j=1}^M U_j.
\end{equation}

By Formula \eqref{ineq_h} and \eqref{ineq_h2}, every $U_i$ is real analytic homeomorphic to an $n$-dimensional convex polytope, and some of them intersect to get a low dimension one. Then we have
\begin{prop}
    Formula \ref{rn_cell} form a finite CW decomposition of $\mathbb{R}^V$.
\end{prop}

From Theorem \ref{diff_all}, we have $\frac{\partial K_i}{\partial u_j}=\frac{\partial K_j}{\partial u_i}$ holds on each $U_i$, and since $K_i$ is $C^1$ continuous, it follows that $\frac{\partial K_i}{\partial u_j}=\frac{\partial K_j}{\partial u_i}$ holds on the entire simply connected $\mathbb{R}_{<0}^V$. Therefore, the differential form $\sum_{i=1}^n K_i du_i$ is a closed form and hence an exact form, and its integral is independent of the choice of path. Thus, $\mathcal{E}_d$ is well-defined on $\mathbb{R}_{<0}^V$. The reason for not including an integration starting point is because the Ricci energy could differ by a constant.

\begin{prop}\label{convex}
    $\mathcal{E}_d$ is a $C^2$ function on $\mathbb{R}_{<0}^V$, and it is a strictly convex.
\end{prop}

\begin{proof}
    We know that $\nabla \mathcal{E}_d(\mathbf{u})=\kappa_d(\mathbf{u})$ is $C^1$ by definition, then $\mathcal{E}_d$ is $C^2$. The matrix $\hess \mathcal{E}_d$ is continuous on $\mathbb{R}_{<0}^V$.

    From Theorem \ref{diff_all}, we know that $\hess \mathcal{E}_d$ is symmetric positive definite on each $U_1,\dots,U_m$. Moreover, since it is globally continuous on $\mathbb{R}_{<0}^V$, it is also symmetric positive definite on the entire space $\mathbb{R}_{<0}^V$. Hence, $\mathcal{E}_d$ is strictly convex.
\end{proof}

Now we prove the main theorem of this paper.

\begin{proof}[Proof of Theorem \ref{main_h}]  
    By Proposition \ref{convex}, we know $\kappa_d|_U$ is injective from variational principle. Denote that
    \[
        K=\left\{\, \mathbf{K}=(K_1,\dots,K_n) \in (-\infty,2\pi)^V \mid \sum_{i=1}^n K_i>2\pi \chi(S) \,\right\}.
    \]

    Then $K$ a bounded open subset of $\mathbb{R}^V$.
    % , and the points on its boundary must have at least one component equal to $2\pi$.
    By the definition of $\kappa_d$ we know it is continuous and $\image \kappa_d \subset K$. By Brouwer's invariance of domain theorem, the map $\kappa_d \colon \mathbb{R}_{<0}^V \to K$ is a continuous injective map between real $n$ dimensional topological manifolds, thus it is an open map. 

    We aim to prove that for any infinite sequence $\left\{\mathbf{u}^{(m)}\right\} \subset \mathbb{R}_{<0}^V$ satisfying
    \begin{equation}\label{infty}
        \lim_{k \to \infty} \mathbf{u}^{(m)}=\mathbf{u}^\infty
        \in [-\infty,0]^n \setminus \mathbb{R}_{<0}^V,
    \end{equation}
    there exists a subsequence $\left\{\mathbf{u}^{(m_i)}\right\}$ such that
    \[
        \lim_{i \to \infty}\kappa_d(\mathbf{u}^{(m_i)}) \in \partial K
    \]

    Since $\mathbb{R}_{<0}^V=\bigcup_{j=1}^M U_j$ is a finite cell decomposition, by the pigeonhole principle, there exists some cell $U_j \subset \tilde \Omega_\mathcal{T}(D_h(\mathcal{T}))$ that contains infinitely many elements from $\left\{\mathbf{u}^{(m)}\right\}$. Without loss of generality, we selected them as a subsequence, but still use the notation $\left\{\mathbf{u}^{(m)}\right\}$ for concise.
    Then we discuss it on the triangulation $\mathcal{T}=\mathcal{T}_j$.

    Consider $\mathbf{u}^\infty=(u^\infty_1,\dots,u^\infty_n)$, if $u^\infty_1=\dots=u^\infty_n=-\infty$, then for any vertex $v_i$, the radius $\lim_{m\to\infty}r_i^{(m)}=0$. That means the area of this hyperbolic polyhedral surface tends to $0$, which satisfies $\kappa_d(\mathbf{u}^{(m)}) \to \partial K$ by Gauss-Bonnet formula $\sum_{i=1}^{n}K_i=2\pi \chi(S)+\area(S)$. If some $u^\infty_k=0$, or $r^{(m)}_k \to +\infty$, then $K^{(m)}_k \to 2\pi$ also tends to boundary.

    Else, there are some but not all components tend to $-\infty$. Let $V_{good} \subsetneq V$ be the non-empty set of vertices corresponding to these components, which are denoted as \emph{good vertices}. The remaining vertices that are not good are denoted as \emph{bad vertices}. Note that the limit of the components of bad vertices may be a negative number or zero as well.

    Firstly, we claim that a triangle $f_{ijk}$ in $\mathcal{T}$ cannot contain exactly one edge with two vertices being bad (note that $v_i$, $v_j$ and $v_k$ may coincide). Otherwise, without loss of generality, assume that $v_i$ is a good while $v_j$ and $v_k$ are bad. During the convergence of $\{\mathbf{u}^{(m)}\}$, since $r_j^{(m)},r_k^{(m)} \ge \epsilon>0$ and $r_i^{(m)} \to 0$, when $m$ is sufficiently large,

    we have $l_{ij}^{(m)}=\cosh r_i^{(m)}\cosh r_j^{(m)}+I_{ij}\sinh r_i^{(m)}\sinh r_j^{(m)} \to r_j^\infty>0$ and $l_{ik}^{(m)} \to r_k^\infty>0$. Moreover, $l_{jk}^{(m)} \to r_j^\infty+r_k^\infty+\delta$ with $\delta \ge 0$. Thus, when $m$ is sufficiently large, the orthogonal circle of face $f_{ijk}$ is no longer compact, which contradicts with $\left\{\mathbf{u}^{(m)}\right\} \subset \tilde \Omega_\mathcal{T}(D_h(\mathcal{T})) \subset \tilde \Omega_\mathcal{T}(\Xi(\mathcal{T}))$.
    
    Secondly, Since $S$ is connected, we can select an edge that connects a good vertex to a bad vertex. All the neighborhood of the bad vertex must be good, otherwise, we could find a neighboring triangle that contradicts the previous argument.

    Finally, by the formula of the inverse distance circle packing, the curvature at this bad vertex converges to $2\pi$ as $k \to \infty$. Namely, $\kappa_d(\mathbf{u}^{(m)})$ converges to the boundary of $K$ up to the subspace topology.
    
    Therefore, we have proved that $\image \kappa_d=K$.
    Otherwise, suppose $\mathbf{K}_0 \in K \setminus\image \kappa_d$, since $\image \kappa_d$ is not empty, we can connect a path $\gamma \colon I \to K$ from the interior of $\image \kappa_d$ to $\mathbf{K}_0$. Because $\image \kappa_d$ is open and simply connected, there exist $s \in I$ such that for any sufficiently large $m$, we have $\gamma(s-1/m) \in \image \kappa_d$ and $\gamma([s,1])\in K\setminus\image\kappa_d$. Let $\mathbf{u}^{(m)}=\kappa_d^{-1}(\gamma(s-1/m))$, then $\mathbf{u}^{(m)}$ converge to $[-\infty,0]^n \setminus \mathbb{R}_{<0}^V$, however, of which any subsequence can not converge to the boundary of $K$, which is a contradiction. Thus, we get that $\kappa_d$ is a bijection.

    The condition of the theorem requires that $\bar{\mathbf K} \in K$, so there exist a unique preimage $\mathbf{u}=\kappa_d^{-1}(\bar{\mathbf K})$, then $\mathbf{B}^{-1}\left([d],\arcth\circ\sech(\mathbf{u})\right)$ is the inversive distance circle packing with discrete curvature $\bar{\mathbf K}$. This is the end of the proof.
\end{proof}

\subsection{Discrete Ricci flow}

The \emph{discrete Ricci flow} of $2$-dimensional polyhedral surface with inversive distance circle packing is defined by
\begin{equation}\label{ricciflow}
    \frac{dr_i}{dt}=-(K_i-\bar K)\sinh r_i
\end{equation}
By Theorem \ref{main_h}, this flow can be extended.

\begin{thm}\label{ricci_flow}
    For any initial $\mathbf{r}(0)=\mathbf{r}_0 \in \mathbb{R}_{>0}$, the ODE \eqref{ricciflow} have the solution $\mathbf{r}(t)$ existing on $\left[0,+\infty\right)$, whose limitation $\mathbf{r}^\infty \coloneqq \lim\limits_{t \to +\infty}\mathbf{r}(t)$ satisfies that on every vertex $v_i$, the discrete curvature of $\mathbf{B}^{-1}([d],\mathbf{r}^\infty)$ is $\bar K$.
\end{thm}

\begin{proof}
    Let $\mathbf{u}(t)\coloneqq \log\tanh\left(\frac12\mathbf{r}(t)\right)$, then ODE \eqref{ricciflow} is
    \[
        \frac{d\mathbf{u}(t)}{dt}=-(\mathbf{K}-\bar{\mathbf{K}}).
    \]

    Define the \emph{normalized} Ricci potential as
    \[
        \mathcal{E}_d(\mathbf{w})=\int ^\mathbf{w} \sum_{i=1}^n (K_i-\bar K_i)\,du_i.
    \]
    Take the derivative with respect to $t$.
    \[
        \begin{aligned}
            \frac{d\mathcal{E}_d(\mathbf{u}(t))}{dt}= & 
            \frac{\partial\mathcal{E}_d}{\partial u_1}\frac{du_1}{dt}+\dots+
            \frac{\partial\mathcal{E}_d}{\partial u_n}\frac{du_n}{dt} \\
            = & -\left(\frac{\partial\mathcal{E}_d}{\partial u_1}(K_1-\bar K_1)+\dots+
            \frac{\partial\mathcal{E}_d}{\partial u_n}(K_n-\bar K_n)\right)\\
            = & -((K_1-\bar K_1)^2+\dots+(K_n-\bar K_n)^2) \le 0.
        \end{aligned}
    \]
    Then $\mathcal{E}_d(\mathbf{u}(t))$ strictly decrease when $\mathbf{K}(t) \ne \bar{\mathbf{K}}$. By Theorem \ref{main_h}, there exist $\mathbf{u}(+\infty) \in U$ such that $\kappa_d(\mathbf{u}(+\infty))=\bar{\mathbf{K}}$, and $\nabla\mathcal{E}_d(\mathbf{u}^\infty)=0$. Thus, $\mathbf{u}^\infty$ is the unique critical point in $U$, and the unique minimum point by convexity. 
    
    Therefore, $\mathcal{E}_d(\mathbf{u}(t))$ has a lower bounded, and the solution can be extended infinitely. When $t \to \infty$, we have $\frac{d\mathcal{E}_d(\mathbf{u}(t))}{dt} \to 0$ and $\mathbf{K}(+\infty) = \bar{\mathbf{K}}$.
\end{proof}

\appendix

\section{Hyperbolic triangle law}

See Figure \ref{fig:hyp3}. Denote the length of a hyperbolic triangle by $x,y,z$ in order, and the opposite angle by $\alpha,\beta,\gamma$. The hyperbolic cosine law is
\begin{equation}\label{cos}
    \begin{aligned}
        \cos \alpha &= \frac{\cosh y \cosh z - \cosh x}{\sinh y \sinh z} \\
        \cosh x &= \frac{\cos \beta \cos \gamma + \cos \alpha}{\sin \beta \sin \gamma}.
    \end{aligned}
\end{equation}

See Figure \ref{fig:hyp6}. Denote the length of a hyperbolic right angle hexagon by $x,c,y,a,z,b$ in order, and the hyperbolic cosine law \cite{mondello2009triangulated} is
\begin{equation}\label{cosh}
    \cosh a= \frac{\cosh y \cosh z + \cosh x}{\sinh y \sinh z}.
\end{equation}

See Figure \ref{fig:hyp4}. Denote the length of a hyperbolic tetragon by $a,x,b,y$ in order, angle two angles at the ends of $x$ are $\frac{\pi}{2}$. Then,
\begin{equation}\label{cosh_4}
    \cosh x= \frac{\sinh a \sinh b + \cosh y}{\cosh a \cosh b}.
\end{equation}

\begin{figure}[ht]
    \centering
    \subfigure{\label{fig:hyp3}}{\includegraphics{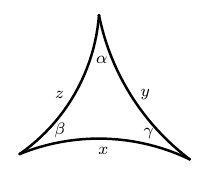}}
    \hspace{2em}
    \subfigure{\label{fig:hyp6}}{\includegraphics{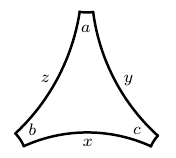}}
    \hspace{2em}
    \subfigure{\label{fig:hyp4}}{\includegraphics{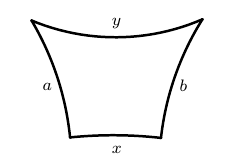}}
    \caption{Hyperbolic cosine laws}
\end{figure}

See Figure \ref{fig:hyp31}. Let $ABCD$ be a hyperbolic tetragon, the angle at $A$, $B$ and $C$ are $\frac{\pi}{2}$, then
\begin{equation}\label{cosh_ratio}
    \cosh AB=\frac{\tanh AD}{\tanh BC} \quad
    \cosh CD=\frac{\sinh AD}{\sinh BC}.
\end{equation}

\begin{figure}[ht]
    \centering \includegraphics{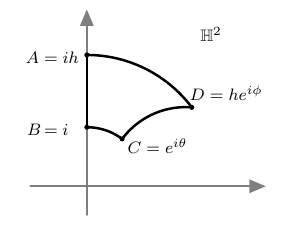}
    \caption{Hyperbolic tetragon with three right angles} \label{fig:hyp31}
\end{figure}

\section{Codes}

The codes for Lemma \ref{dfdF} is below. The software is Mathematica 8.0.4.0.
\begin{verbatim}
X=Numerator@
  Factor[u^2 w^2+v^2 x^2+y^2 z^2-u^2-v^2-w^2-x^2-y^2-z^2+1
    -2 (u v w x+u w y z+v x y z-v w y-u x y-u v z-w x z)/.{
    u->1/Sqrt[1-p^2] 1/Sqrt[1-q^2]
    +a p/Sqrt[1-p^2] q/Sqrt[1-q^2],
    v->1/Sqrt[1-r^2] 1/Sqrt[1-q^2]
    +b r/Sqrt[1-r^2] q/Sqrt[1-q^2],
    w->1/Sqrt[1-r^2] 1/Sqrt[1-s^2]
    +c r/Sqrt[1-r^2] s/Sqrt[1-s^2],
    x->1/Sqrt[1-p^2] 1/Sqrt[1-s^2]
    +d p/Sqrt[1-p^2] s/Sqrt[1-s^2],
    y->1/Sqrt[1-q^2] 1/Sqrt[1-s^2]
    +e q/Sqrt[1-q^2] s/Sqrt[1-s^2],
    z->1/Sqrt[1-p^2] 1/Sqrt[1-r^2]
    +F p/Sqrt[1-p^2] r/Sqrt[1-r^2]}];
Y=a^2+b^2+c^2+d^2+e^2+f^2-a^2 c^2-b^2 d^2-e^2 f^2-1 +
  2(a d e+b c e+a b f+c d f+a b c d+a c e f+b d e f);
p1=(p/.Solve[D[X,p]==0,p][[1]]);
Factor[X/.p->p1]
Factor[D[X,q]/.p->p1/.F->f]
Factor[D[Y,a]/D[Y,f]-D[X,a]/D[X,F]/.p->p1/.F->f]
Factor[D[Y,e]/D[Y,f]-D[X,e]/D[X,F]/.p->p1/.F->f]
\end{verbatim}

The output is one fraction with factor $Y(a,b,c,d,e,F)$ and three fractions with factor $Y(a,b,c,d,e,f)$ in their numerators.

\bibliographystyle{alpha}
\bibliography{ref}

\end{document}